\documentclass[10pt]{amsart}
\usepackage{amssymb,amsmath,amsfonts,amscd,mathrsfs, bbm, xcolor}

\usepackage{a4wide}
\usepackage[all,hyperref,numberbysection]{bi-discrete}
\usepackage{upgreek}
\usepackage{microtype}
\numberwithin{equation}{section}

\begin{document}

\author[K. Kim]{Kyeongbae Kim}
\address{Institute for Applied Mathematics, University of Bonn, Endenicher Allee 60, 53115, Bonn, Germany}
\email{kkim@uni-bonn.de}

\author[S. Nowak]{Simon Nowak}
\address{Fakult\"at f\"ur Mathematik, Universit\"at Bielefeld, Postfach 100131, D-33501 Bielefeld, Germany}
\email{simon.nowak@uni-bielefeld.de}

\author[Y. Sire]{Yannick Sire}
\address{\noindent
Department of Mathematics,
Johns Hopkins University, 3400 N. Charles Street, Baltimore, MD 21218, USA}
\email{ysire1@jhu.edu}

\title{Fine regularity of fractional harmonic maps and applications}

\begin{abstract}
In this paper, we derive several regularity results for harmonic mappings into Euclidean spheres associated with rather general energies related to fractional Sobolev spaces. These maps generalize families of maps introduced by Da Lio, Rivi\`ere and Schikorra and are related to harmonic maps with free boundaries. In our context, there is in general no monotonicity formula, which prevents the use of some classical methods.  Despite this limitation, under natural assumptions on a Gagliardo-type energy, we succeed in proving a variety of small energy regularity results and improve on known results, even in the isotropic case for which some monotonicity formula is available. To this end, we exploit recent developments in the regularity theory of nonlocal equations and as a by-product, we explain how these results apply to classes of harmonic maps with free boundary and lead to new potential-theoretic estimates. As another application, we obtain higher differentiability results for the fractional harmonic map heat flow.
\end{abstract}

\maketitle

\tableofcontents

%% ------------------------------------------------------------

\section{Introduction}
In this paper, we derive several regularity results for nonlocal systems of fractional harmonic mappings into Euclidean spheres. We also apply those to new regularity estimates to derive results for harmonic mappings with free boundaries and the Plateau flow introduced by Struwe recently in \cite{plateau}. As a by-product of the fine regularity properties we develop, we obtain some potential-theoretic estimates in terms of the data which are new in our general context but as well as for harmonic maps with free boundary. 

\subsection*{The setup}
We consider maps $u$ from $\Omega \subset \bbR^n$ with $n \geq 1$ into $\bbR^N$ with $N\geq 2 $ as well such that for almost every $x \in \Omega$, $u(x) \in \mathbb S^{N-1}$, the round sphere in $\mathbb R^N$.  Define now the following operator  

\begin{equation*}
	\mathcal{L}_su(x)\coloneqq \,\text{p.v.}\int_{\bbR^n}\frac{u(x)-u(y)}{|x-y|^s}\frac{a(x-y)}{|x-y|^{n+s}}\,dy,
\end{equation*}
where $s \in (0,1)$ and the scalar function $a$ is measurable and satisfies 
\begin{equation}\label{cond.a}
	\Lambda^{-1}\leq a(\xi)=a(-\xi)\leq \Lambda
\end{equation}for some constant $\Lambda\geq1$. In the previous expression, $p.v.$ stands for the principal value in the Cauchy sense. 
Writing 
\begin{equation*}
	d_su(x,y)\coloneqq \frac{u(x)-u(y)}{|x-y|^s}
\end{equation*}
and 
\begin{equation*}
	\mathcal{C}\Omega\coloneqq (\Omega^c\times \Omega^c)^c, 
\end{equation*}
for $p \in (1,\infty)$ consider the following energy 
\begin{align}\label{defn.wicale}
	\widetilde{\mathcal{E}}(u;\Omega)\coloneqq \iint_{\mathcal{C}\Omega}|d_su|^p\frac{\,dx\,dy}{|x-y|^n},
\end{align}
which is defined on the Sobolev space 
\begin{equation*}
	\widehat{W}^{s,p}(\Omega;\bbR^N)\coloneqq\{u\in L^2(\Omega;\bbR^N)\,:\, \widetilde{\mathcal{E}}(u,\Omega)<\infty\}. 
\end{equation*}
Our operator $\mathcal L_s$ appears as a variable coefficient version of the generator of the energy $\widetilde{\mathcal{E}}$ for $p=2$. 

We are interested in the regularity of weak solutions $u:\Omega \to \bbR^N$ such that $|u|=1$ almost everywhere in $\bbR^n$\footnote{Here we note that for any $u=(u^i)_i:\bbR^n\to \bbR^N$,  
	\begin{equation*}
		|u|^2=\sum_{i}^N(u^i)^2.
\end{equation*}}, to the following equation
\begin{equation}\label{eq.mainIntro}
	\mathcal{L}_su=|\mathbf{d}_{s,a}u|^2u+\mu\quad\text{in }\Omega,
\end{equation}
where we denote here 
\begin{equation*}
	|\mathbf{d}_{s,a}u|^2(x)\coloneqq \int_{\bbR^n}\frac{|u(x)-u(y)|^2}{|x-y|^{2s}}\frac{a(x-y)}{|x-y|^n}\,dy
\end{equation*}
and assuming some integrability for the term $\mu.$ In particular, weak solutions here belong to the fractional Sobolev Space 
\begin{equation*}
	\widehat{H}^s_{00}(\Omega;\bbR^N)\coloneqq\{u\in \widehat{H}^s(\bbR^n;\bbR^N)\,:\,u\equiv 0\text{ on }\bbR^n\setminus \Omega\}.
\end{equation*}

The equation \eqref{eq.mainIntro} is a harmonic map type equation which generalizes the so-called fractional harmonic maps that we now describe more in details. In the seminal papers \cite{DLR1,DLR2}, Da Lio and Rivi\`ere introduced the concept of half-harmonic maps, whose definition follows. 

\begin{definition}\label{defHalfHarm}
	A map $u : \mathbb{R}^n \to \mathbb S^{N-1}$ belonging to $H^{1/2}(\mathbb R^n, \mathbb S^{N-1})$
	is called a \emph{half-harmonic map} if it is a critical point of the fractional energy functional
	\[
	E_{1/2}(u) \;=\; \frac{1}{2} \int_{\mathbb{R}^n} \int_{\mathbb{R}^n} 
	\frac{|u(x) - u(y)|^2}{|x-y|^{n+1}} \, dx \, dy .
	\]
	
	Equivalently, $u$ satisfies in the distributional sense, the Euler–Lagrange equation
	\[
	(-\Delta)^{1/2} u(x) \;\perp\; T_{u(x)}\mathbb S^{N-1} ,
	\]
	where $(-\Delta)^{1/2}$ is the Fourier multiplier of symbol $|\xi|$, and $T_{u(x)}\mathbb S^{N-1}$ is the tangent space to $\mathbb S^{N-1}$ at $u(x)$.
\end{definition}

The terminology ``half-harmonic'' comes from a dimensional reduction principle.  Indeed, the following equivalence holds between the Dirichlet energy of a function $U$ in the half-space and the $H^{1/2}$-energy of its trace $u$ on $\partial \mathbb R^{n+1}_+$
\[
\int_{\mathbb{R}^{n+1}_+} |\nabla U|^2 \, dx\,dt
\quad \longleftrightarrow \quad
\int_{\mathbb{R}^n} \int_{\mathbb{R}^n} \frac{|u(x) - u(y)|^2}{|x-y|^{n+1}} \, dx \, dy ,
\]
in the sense that any harmonic function $U$ such that $\int_{\mathbb{R}^{n+1}_+} |\nabla U|^2 \, dx\,dt$ is finite has a trace $u$ such that $ \int_{\mathbb{R}^n} \int_{\mathbb{R}^n} \frac{|u(x) - u(y)|^2}{|x-y|^{n+1}} \, dx \, dy < \infty$, and equivalently any function $u$ with finite $H^{1/2}$-energy in $\mathbb R^n$ can be extended harmonically to the the upper half-space to a function $U$ with finite Dirichlet energy. This principle applies as well to maps targeted into spheres. In \cite{millotSire}, Millot and the third author proved that any weak half-harmonic map satisfies in the weak sense the following equation

\begin{equation*}
	(-\Delta)^{1/2}u=|\mathbf{d}_{1/2,1}u|^2u\quad\text{in }\Omega. 
\end{equation*}

The equation \eqref{eq.mainIntro} is then a generalization of the previous equation, allowing variable kernels, other possible orders of differentiation and non-homogeneous right-hand sides. As far as the order of the equation is concerned, the papers \cite{MilPegSch21,Millot-Sire-Yu} deal with regularity results as well for operators like $(-\Delta)^s$, i.e.\ multipliers of symbol $|\xi|^{2s}$ for $s \in (0,1)$. For nonlinear operators in the principal part of the equation, we refer to the papers of Da Lio and Schikorra \cite{DLS1,DLS2,armin1}. 
Half-harmonic maps have also an important geometric interpretation. As in the case of classical harmonic maps, they parametrize (possibly branched) minimal immersions with free boundary (see \cite{millotSire} and references therein for such a connection); they also enjoy conformal properties. Indeed, exploiting the extension principle mentioned above, the energy in Definition \ref{defHalfHarm} is conformally invariant for $n=1$ in the sense that it is invariant under any Moebius transformation of $\mathbb R^2_+$ which leaves invariant the boundary $\bbR$. Hence, it is important to remark that half-harmonic maps from the real line into spheres (or more generally any closed manifolds) should enjoy full regularity and should be in particular H\"older continuous everywhere on the domain. This has been indeed proved in \cite{DLR1}. As a matter of fact, in our analysis, equation \eqref{eq.mainIntro} is {\sl critical} only for $n=1$ and $s=1/2$ and {\sl supercritical} otherwise. In this latter case, the only hope in general is to obtain {\sl partial regularity} of weak solutions. Classical harmonic maps are instrumental in many areas of mathematics and giving an account of the theory is beyond the scope of this paper. We refer the reader to the book by Lin and Wang \cite{MR2431658} for a recent detailed account. In the following sections, we will also refer to some classical results and references, for which our results can be seen as a generalization.

\subsection*{The results}
This section is devoted to the main results of this paper. Later, we will also draw some consequences for more classical objects like harmonic maps with free boundary and the Plateau flow of Struwe. We would like to emphasize the three following points:
\begin{itemize}
	\item We are interested in considering {\sl approximate} fractional harmonic maps (i.e. $\mu$ not identically zero) for compactness purposes but to also obtain results for time-dependent problems. 
	\item There is a dichotomy between critical equations (i.e. $n=1,s=1/2$) and supercritical ones. In particular, all the smallness assumptions disappear in the critical case, providing new results in our variable setting and new proofs of existing results. In the supercritical setup, it is not possible in general to remove the smallness assumption. 
	\item A very important feature of the classical harmonic map system is the monotonicity of the rescaled energy $\frac{1}{r^{n-2}}\int_{B_r} |\nabla u|^2$ in terms of the radius $r>0$. Our energies are never monotone at the fractional level and only enjoy some type of monotonicity in the special case of $\mathcal L_s =(-\Delta)^s$ thanks to the Caffarelli-Silvestre extension theorem. Monotonicity implies in particular two important properties: If one controls one scale, then one controls also every smaller scale; blow-up limits exist and a dimension reduction a la Federer can be used to prove strong Hausdorff dimension bounds on the singular sets of partially regular maps. Except in the critical case (where monotonicity is actually not needed), our results are then weaker in such a general framework than the ones available in the literature for similar equations (e.g. \cite{MilPegSch21}). Nevertheless, our standing assumptions are quite general and natural in this setting.  
\end{itemize}

We define
\begin{equation*}
	\widehat{W}^{s,p}_{00}(\Omega;\bbR^N)\coloneqq\{u\in \widehat{W}^{s,p}(\bbR^n;\bbR^N)\,:\,u\equiv 0\text{ on }\bbR^n\setminus \Omega\}.
\end{equation*}

As is customary by now in the regularity theory of nonlocal equations, quantitative estimates a la Campanato are at the cornerstone of the proofs of regularity. %We state below a potential theoretic excess decay which is paramount to our various estimates. 

We would like also to emphasize that our methods are so general and flexible that in some cases, we can also treat in a very unified way a quasilinear version of our maps. More precisely, we will consider operators of so-called fractional $p$-Laplacian type. These are of the form
\begin{equation*}
	\mathcal{L}_{s,p}u(x)\coloneqq \,\text{p.v.}\int_{\bbR^n}|d_su|^{p-2}d_s u\frac{a(x,y)}{|x-y|^{n+s}}\,dy,
\end{equation*}
where $p \in (1,\infty)$ and $a$ is measurable, symmetric in $x$ and $y$, and satisfies $\Lambda^{-1} \leq a(x,y) \leq \Lambda$ for some $\Lambda \geq 1$. In the classical case, $p$-harmonic maps have been investigated in a large amount of works, see e.g. \cite{Hardt-Lin87}. In this case, the criticality holds whenever the index $p$ equals the dimension. The maps we are considering have been introduced by Da Lio and Schikorra in particular and some regularity results have been obtained in \cite{armin1} and \cite{DLS1}.

We now give a precise notion of a weak solution.
\begin{definition}
	Let $s \in (0,1)$, $p \in (1,\infty)$ and let us assume $\mu\in L^{p_*}(\Omega;\bbR^N)$, where $p_*=\max\{np/(n(p-1)+sp),1\}$. We say that $u$ is a weak solution to 
	\begin{equation}\label{eq.main}
		\mathcal{L}_{s,p}u=|\mathbf{d}_{s,a}u|^pu+\mu\quad\text{in }\Omega,
	\end{equation}
	if $u\in \widehat{W}^{s,p}(\Omega;\bbR^N)$ with $|u|\equiv 1$ on $\bbR^n$ satisfies
	\begin{align*}
		\iint_{\mathcal{C}\Omega}|d_su|^{p-2}d_sud_s\phi\frac{a(x,y)}{|x-y|^n}\,dx\,dy&=\int_{\Omega}\left(\int_{\bbR^n}|d_su|^p \frac{a(x,y)}{|x-y|^{n}}\,dy\right)u(x)\phi(x)\,dx\\
		&\quad+\int_{\Omega}\mu\phi
	\end{align*}
	for any $\phi\in \widehat{W}^{s,p}_{00}(\Omega)$, where
	\begin{equation*}
		|\mathbf{d}_{s,a}u|^p(x)\coloneqq \int_{\bbR^n}\frac{|u(x)-u(y)|^p}{|x-y|^{sp}}\frac{a(x,y)}{|x-y|^n}\,dy.
	\end{equation*}
\end{definition}

We now introduce our main result.
\subsection{Main results in the critical case} We introduce regularity results for the critical case $n=sp$. First, we provide the following H\"older regularity result.
\begin{theorem}\label{thm0}
	Under the above assumptions, let $u$ be a weak solution to 
	\begin{equation*}
		\mathcal{L}_{s,p}u=|\mathbf{d}_{s,a}u|^pu+\mu\quad\text{in }\Omega
	\end{equation*}
	with $p=\frac{n}{s}$. Suppose $B_{2R}(x_0)\Subset \Omega$. Then there is a small constant $\delta=\delta(n,s,N,\Lambda)$ such that if $\mu\in L^{{\gamma}}(B_{2R}(x_0))$ with $\frac{1}{p-1}\left(sp-\frac{n}\gamma\right)<\delta$, we have $u\in C^{\frac{1}{p-1}\left(sp-\frac{n}\gamma\right)}(B_R(x_0);\bbR^N)$.
\end{theorem}
\begin{remark}
	We note from Lemma \ref{lem.fracdecay} that any weak solution of homogeneous fractional p-Laplace type systems with $p=\frac{n}{s}$ is $\alpha$-order H\"older continuous for some constant $\alpha=\alpha(n,s,\Lambda,N)\in(0,1)$. Indeed, $\delta$ in Theorem \ref{thm0} is given by $\delta=\min\left\{\alpha,\frac{sp}{p-1}\right\}$.
\end{remark}
In the linear case when $p=2$, we obtain the following higher regularity results.
\begin{theorem}\label{cor2}
	Let $n=1$ and let $u$ be a weak solution to 
	\begin{equation*}
		\mathcal{L}_{1/2} u=|\mathbf{d}_{1/2,a}u|^2u+\mu\quad\text{in }\Omega,
	\end{equation*}
	with $a(x,y)=a(x-y)$ satisfying \eqref{cond.a}. If $\mu\in L^{\gamma}(B_{2R}(x_0);\bbR^N)$ with $\gamma \in (1,\infty)$, then $u\in C^{1-\frac1\gamma}(B_R(x_0);\bbR^N)$. In addition, if $a=1$ and $\mu \in L^2(B_{2R}(x_0))$, then $u\in W^{1,2}(B_R(x_0);\bbR^N)$.
\end{theorem}

We note that Theorem \ref{cor2} is new already for $a \equiv \textnormal{const}$, that is, in the case of approximate half-harmonic maps.

Our methods actually allow to obtain fine potential-theoretic estimates in terms of fractional maximal functions (see \eqref{maxfun} for their precise definition), as the following result shows. These estimates are in the spirit of nonlinear potential theory (see \cite{KuMiU,KuMiG}) and imply Theorem \ref{thm0} in a rather straightforward fashion. 

\begin{theorem}\label{pot_intro}
	Let $p=\frac{n}{s}$ and $\gamma>1$. Fix $\vartheta< \min\left\{\alpha,\frac{sp}{p-1}\right\}$, where the constant $\alpha=\alpha(n,s,\Lambda,N)$ is determined in Lemma \ref{lem.fracdecay}.
	Let $u\in \widehat{W}^{s,p}(B_{2R}(x_0);\bbR^N)$ with $|u|\equiv 1$ on $\bbR^n$ be a weak solution to 
	\begin{equation}\label{goal.decayu2}
		\mathcal{L}_{s,p}u=|\mathbf{d}_{s,a}u|^pu+\mu\quad\text{in }B_{R}(x_0).
	\end{equation}
	Then there is a sufficiently small radius $R_0=R_0(n,s,\Lambda,N,\vartheta,\gamma)$ such that for all $R \leq R_0$, we have
	\begin{align*}
		\sup_{0<r\leq R}r^{-\vartheta}\dashint_{B_r(x_0)}|u-(u)_{B_{r}(x_0)}|\,dx\leq c\left(\epsilon/R^{\vartheta}+M_{\gamma(sp-\vartheta(p-1)),R}(|\mu|^{\gamma})(x_0)^{\frac{1}{\gamma(p-1)}}\right)
	\end{align*}
	for some constant $c=c(n,s,\Lambda,N,\vartheta,\gamma)$.
\end{theorem}

\subsection{Main results in the supercritical case} We introduce $\varepsilon$-regularity results for linear problems with translation-invariant kernel coefficients in the supercritical case when $2s<n$. Indeed, this is the only interesting case left to be treated, since for $2s>n$, by Morrey's embedding weak solutions are continuous, so that the regularity becomes trivial. See e.g.\ the paper \cite{Millot-Sire-Yu}.
\begin{theorem}\label{thm1}
	Let $n>2s$ and let $u$ be a weak solution to 
	\begin{equation*}
		\mathcal{L}_su=|\mathbf{d}_{s,a}u|^2u+\mu\quad\text{in }\Omega,
	\end{equation*}
	where $a(x,y)=a(x-y)$ satisfies \eqref{cond.a}. Suppose $B_{2R}(x_0)\Subset \Omega$ and
	\begin{equation}\label{ass.thm1}
		\sup\limits_{B_{\rho}(z)\subset B_{2R}(x_0)}\rho^{2s-n}\widetilde{\mathcal{E}}(u;B_\rho(z))\leq \epsilon
	\end{equation}
	for some $\epsilon\in(0,1]$. If $\mu\in L^\gamma(B_R(x_0);\bbR^N)$ with $\gamma\neq \frac{n}{2s-1}$ and the constant $\epsilon=\epsilon(n,s,N,\Lambda,\gamma)$ is sufficiently small, then $u\in C^{2s-\frac{n}\gamma}(B_R(x_0);\bbR^N)$. Moreover, if $s>\frac12$, $\mu$ belongs to the Lorentz space $L^{\frac{n}{2s-1},1}(B_R(x_0);\bbR^N)$ and $\epsilon=\epsilon(n,s,N,\Lambda)$ is sufficiently small, then we have $u\in C^1(B_R(x_0))$.
\end{theorem}

\begin{remark}
	We refrain to derive our results in the most general setting. For instance, by invoking comparison techniques from e.g.\ \cite{MeV,ByunKimCC,DefMinNow25}, we expect that in all our main results where we assume translation invariance, the coefficient $a$ could satisfy weaker regularity properties, such as a small BMO norm or a VMO property.
\end{remark}

Next, we establish a higher Sobolev regularity result of Calder\'on-Zygmund type in the case of the fractional Laplacian.
\begin{theorem}\label{thm3}
	Let $n>2s$ and let $u$ be a weak solution to 
	\begin{equation*}
		(-\Delta)^su=|\mathbf{d}_{s}u|^2u+\mu\quad\text{in }\Omega.
	\end{equation*}
	Suppose $B_{2R}(x_0)\Subset \Omega$ and \eqref{ass.thm1} with a small constant $\epsilon\in(0,1]$. For any $\gamma\geq2$, if $n<2s\gamma$, $\mu\in L^{\gamma}(B_{2R}(x_0);\bbR^N)$, and  $\epsilon=\epsilon(n,s,N,\gamma)$ is sufficiently small, then $u\in W^{2s,\gamma}(B_R(x_0);\bbR^N)$. 
\end{theorem}
\begin{remark}We point out that we also prove higher Sobolev regularity of solutions to \eqref{eq.main} with $p=2$ and $a(x,y)=a(x-y)$ (see Theorem \ref{thm2} below).
\end{remark}

An important aspect of our results, actually their proof, is that since we cannot use the Caffarelli-Silvestre extension, we are forced to deal only with the boundary equation. In particular, we need to deal with the nonlocal nature of the equation at every stage of the proof. Due to these complications, in the general case, our smallness hypotheses may not be optimal. However, it is clear that in the case of the fractional Laplacian, we recover some results by Pegon, Millot and Schikorra \cite{MilPegSch21} providing an alternative proof and actually, allowing non homogeneous right-hand side data, which is important for compactness issues. 

%\marginpar{Y: I commented out the remark on the singular set}
%\begin{remark}
%	Of course, using standard covering arguments, our statements imply some measure-theoretic estimates on the singualr set of the maps. Once again, the lack of monotonicity and our general framework, prevent us to implement a dimension reduction, which would lead to much better estimates of the set of singualr points of the map. We refrain to state such suboptimal results. 
%\end{remark}

\subsection{Results for time-dependent equations}

We now apply our result to the heat flow of the fractional harmonic maps. First, we give the definition of weak solutions.
\begin{definition}\label{defn.flow}
	We say that $u=u(x,t)$ is a weak solution to
	\begin{equation*}
		\partial_tu+(-\Delta)^su=|\mathbf{d}_{s}u|^2u\quad\text{in }\Omega_T\coloneqq \Omega\times (0,T],
	\end{equation*}
	if $u\in L^\infty(0,T;\widehat{H}^{s}(\Omega;\bbR^N))\cap H^1(0,T;L^2(\Omega;\bbR^N))$ with $|u|\equiv 1$ on $\bbR^n\times (0,T]$ satisfies
	\begin{align*}
		&\int_{0}^{T}\int_{\Omega}\partial_tu\phi\,dx\,dt+\int_{0}^{T}\int_{\mathcal{C}\Omega}d_sud_s\phi\frac{1}{|x-y|^n}\,dx\,dy\,dt\\
		&=\int_{0}^{T}\int_{\Omega}\left(\int_{\bbR^n}|d_{s}u(x,y,t)|^2 \frac{1}{|x-y|^{n}}\,dy\right)u(x,t)\phi(x,t)\,dx\,dt
	\end{align*}
	for any $\phi\in L^2(0,T;\widehat{H}^s_{00}(\Omega;\bbR^N))$.
\end{definition}
\begin{remark}
	In \cite{SchSirWan17}, the existence theory for the corresponding initial boundary value problem is established. 
\end{remark}
We now present a higher differentiability result of the solution to the heat flow of the fractional harmonic map.
\begin{theorem}\label{thm4}
	Let $u$ be a weak solution to 
	\begin{equation*}
		\partial_tu+(-\Delta)^su=|\mathbf{d}_{s}u|^2u\quad\text{in }\Omega_T.
	\end{equation*}
	Fix $\sigma\in(0,\min\{1,2s\})$. There is a small constant $\epsilon=\epsilon(n,s,N,\sigma)$ such that if 
	\begin{equation}\label{ass.thm4}
		\begin{aligned}
			\sup\limits_{B_{\rho}(x)\subset B_{2R}(x_0)}\sup_{t\in I_{2R}(t_0)}\rho^{2s-n}\widetilde{\mathcal{E}}(u(\cdot,t);B_\rho(x))\leq \epsilon
		\end{aligned}
	\end{equation}
	then $u\in L^2(I_R(t_0);W^{\sigma,2}(B_R(x_0);\bbR^N))$. In particular, if $n<4s$, then we can take $\sigma=2s$.
\end{theorem}
As a corollary, we derive the same result as in \cite[Proposition 1.5]{Wet22} by a different approach.
\begin{corollary}\label{cor3}
	Let $n=1$, $s=1/2$ and let $u$ be a weak solution to 
	\begin{equation*}
		\partial_tu+(-\Delta)^{1/2}u=|\mathbf{d}_{1/2}u|^2u\quad\text{in }\Omega_T.
	\end{equation*}
	Then $u\in L^2(I_R(t_0);W^{1,2}(B_R(x_0);\bbR^N)$ for any $I_R(t_0)\times B_R(x_0)\Subset \Omega_T$.
\end{corollary}

\subsection{Applications to harmonic maps with free boundary.} As explained before, one of our main motivations consists of obtaining {\sl new} estimates for harmonic maps with free boundary, which are relevant systems of harmonic maps with geometric Neumann boundary conditions. In this section, we state several consequences of our results for such PDEs and refer the reader to Section \ref{sec:applications} for more details.  We start with the definition of weak harmonic maps with free boundary. It has been proved in \cite{millotSire} that an (almost-)harmonic map with free boundary into the sphere satisfies (in the distributional sense)
\begin{equation}\label{defFBIntro}
\begin{cases}
\Delta u = 0 & \text{in $\Omega \subset \mathbb{R}^{n+1}_+$}\,,\\[8pt]
\displaystyle \frac{\partial u}{\partial\nu} \,+\mu\, \bot  \ T_u\,\mathbb{S}^{N-1} & \text{in $H^{-1/2}(\partial^0 \Omega)$}\,,
\end{cases}
\end{equation}
where $\partial^0 \Omega \subset \partial \mathbb{R}^{n+1}_+$.
The previous theorems have the following consequences for \eqref{defFBIntro}. 
\begin{theorem}\label{thFBIntro}
Assume $n=1$ and let $u$ be a weak solution of \eqref{defFBIntro}. If $\mu\in L^{\gamma}(B_{2R}(x_0);\bbR^N)$ with $\gamma \in (1,\infty)$, then $\text{Tr}(u)\in C^{1-\frac1\gamma}(B_R(x_0);\bbR^N)$. If  $\mu \in L^2(B_{2R}(x_0))$, then $\text{Tr}(u)\in W^{1,2}(B_R(x_0);\bbR^N)$, where $\text{Tr}(u)$ is the trace in $H^{1/2}$ of $u$ on $\partial^0 \Omega$. 
\end{theorem}

\section{Preliminaries }
 In this paper, we usually denote a universal constant that is equal to or greater than 1 as $c$. Moreover, when the constant $c$ depends on some parameters such as $n,s$ and $\Lambda$, we write $c=c(n,s,\Lambda)$. We also denote $B_R(x_0)$ the ball in $\bbR^n$ with center $x_0$ and radius $R>0$.

We now introduce some geometric and functional notation. Let us write $B_R(x_0)$ to mean the ball in $\bbR^n$ with center $x_0$ and radius $R>0$. 

Let us define
\begin{equation*}
    L^p_{\mathrm{od}}(\Omega)\coloneqq\{F:\mathcal{C}\Omega\to\bbR\,:\, F(x,y)=-F(y,x)\quad\text{and}\quad \|F\|_{L^p_{\mathrm{od}}}<\infty\},
\end{equation*}
where
\begin{align*}
    \|F\|_{L^p_{\mathrm{od}}(\Omega)}\coloneqq \left(\iint_{\mathcal{C}\Omega}\frac{|F(x,y)|^p}{|x-y|^n}\,dx\,dy\right)^{\frac1p}.
\end{align*}
We recall that for any $F\in L^{p}_{\mathrm{od}}(\Omega)$ and $ G\in L^{p'}_{\mathrm{od}}(\Omega)$,
\begin{align}\label{defn.odot}
    (F\odot G)(x)\coloneqq\int_{\bbR^n}\frac{F(x,y)G(x,y)}{|x-y|^n}\,dy.
\end{align}
In addition, we say that $F\in L^p_{\mathrm{od}}(\Omega)$ satisfies
\begin{align*}
    -\divergence_s F=g \quad\text{in }\Omega,
\end{align*}
if 
\begin{align*}
    \int_{\bbR^n}F \odot d_s\psi=\int_{\bbR^n}g\psi
\end{align*}
for any $\psi\in \widehat{W}^{s,p'}_{00}(\Omega)$. Note that these notations were introduced by \cite{MazSch18}.

Now we recall a suitable tail notation, see e.g.\ \cite{DicKuuPal16}:
\begin{align*}
    \mathrm{Tail}_q(u;B_{R}(x_0))\coloneqq\left(R^{sp}\int_{\bbR^n\setminus B_R(x_0)}\frac{|u(y)|^q}{|y-x_0|^{n+sp}}\,dy\right)^{\frac1q},
\end{align*}
for any $q\geq 1$. With this notation, we now introduce an appropriate excess functional 
\begin{equation}\label{defn.cale}
\begin{aligned}
\mathcal{E}(u;B_R(x_0))&\coloneqq R^{sp}\dashint_{B_{R}(x_0)}\int_{B_{R}(x_0)}|d_su|^p\frac{\,dx\,dy}{|x-y|^{n}}+\mathrm{Tail}_q(u-(u)_{B_R(x_0)};B_{R}(x_0))^p.
\end{aligned}
\end{equation}
Since we will deal with the parameter $q$ which is very close to $p-1$, we now fix the constant $q$ 
\begin{equation}\label{cond.q}
    p-1<q\leq \max\{(p+1)/2,(2p-1)/2\}.
\end{equation}
\subsection{Fractional Sobolev and BMO spaces}
First, we recall the Sobolev Poincar\'e inequality given in \cite[Theorem 6.7]{DinPalVal12} and \cite[Lemma 4.7]{Coz17}.
\begin{lemma}\label{lem.spi}
    Let $u\in W^{s,p}(B_R(x_0))$ with $s\in(0,1)$ and $p\geq1$. Let us choose $p^*\in [1,np/(n-sp)]$ if $n<sp$ and $p^*\in[1,\infty)$ if $n\geq sp$.
    Then there is a constant $c=c(n,s,p,p^*)$ such that
    \begin{align*}
        \left(\dashint_{B_R(x_0)}\left|u-(u)_{B_R(x_0)}\right|^{p^*}\,dx\right)^{\frac1{p^*}}\leq cR^{s-n/p}[u]_{W^{s,p}(B_R(x_0))}.
    \end{align*}
    In addition, if $u\in \widehat{W}_{00}^{s,p}(B_R(x_0))$, then we have 
   \begin{align*}
        \left(\dashint_{B_R(x_0)}\left|u\right|^{p^*}\,dx\right)^{\frac1{p^*}}\leq cR^{s-n/p}[u]_{W^{s,p}(B_{2R}(x_0))}.
    \end{align*}
\end{lemma}

We now recall the well-known property of BMO space
\begin{equation}\label{prop.bmo}
    [u]_{BMO(B_R)}\coloneqq \sup_{B_r(z)\subset B_R}\dashint_{B_r(z)}|u-(u)_{B_r(z)}|\eqsim \sup_{B_r(z)\subset B_R}\left(\dashint_{B_r(z)}|u-(u)_{B_r(z)}|^p\right)^{\frac1p}
\end{equation}
for any $p\geq1$ (see \cite[Corollary 6.23]{GiaMar12} for the proof of \eqref{prop.bmo}).

We next provide Gagliardo-Nirenberg interpolation inequalities between fractional Sobolev spaces and BMO space (see \cite[Lemma 15.7]{BreMir21} or \cite[Theorem 7]{Van23} for the proof of the following lemma).
\begin{lemma}\label{lem.gagnir} 
Let $0<s<1$, $p\geq1$ and $\theta\in(0,1)$. Then we have
\begin{equation*}
    [f]_{W^{\theta s,p/\theta}(B_1)}\leq c[f]^{1-\theta}_{BMO(B_1))}[f]^{\theta}_{W^{s,p}(B_1)}
\end{equation*}
for some constant $c=c(n,s,p,\theta)$.
\end{lemma}

\subsection{Basic inequalities}
Next, we observe the relation between $\mathcal{E}$ and $\widetilde{\mathcal{E}}$ which are defined in \eqref{defn.wicale} and \eqref{defn.cale}, respectively.
\begin{lemma}\label{lem.reltildeee}
    Let $u\in \widehat{W}^{s,p}(B_{2R}(x_0))$ with $|u|\equiv 1$ on $\bbR^n$. Then there is a constant $c=c(n,s,\Lambda,p,q)$ such that 
    \begin{align}\label{ineq1.reltildeee}
        \mathcal{E}(u;B_{R}(x_0))\leq cR^{sp-n}\widetilde{\mathcal{E}}(u;B_{2R}(x_0)).
    \end{align} 
    and
    \begin{align}\label{ineq2.reltildeee}
        R^{sp-n}\widetilde{\mathcal{E}}(u;B_R(x_0))\leq c\left[\mathcal{E}(u;B_{2R}(x_0))+\mathcal{E}(u;B_{2R}(x_0))^{\frac{q}{p}}\right]
    \end{align}
\end{lemma}
\begin{proof}
We may assume $x_0=0$. After a few simple computations together with H\"older's inequality, we have
\begin{align*}
    I&\coloneqq\left(R^{sp}\int_{\bbR^n\setminus B_R}\frac{|u(y)-(u)_{B_R}|^q}{|y|^{n+sp}}\,dy\right)^{\frac{p}q}\\
    &\leq cR^{sp}\int_{\bbR^n\setminus B_R}\frac{|u(y)-(u)_{B_R}|^p}{|y|^{n+sp}}\,dy\\
    &\leq c\left(R^{sp}\int_{\bbR^n\setminus B_{2R}}\frac{|u(y)-(u)_{B_R}|^pp}{|y|^{n+sp}}\,dy+\dashint_{B_{2R}}|u-(u)_{B_R}|^p\,dy\right)
\end{align*}
for some constant $c=c(n,s,p)$.
In light of the Poincar\'e inequality and the fact that $|y-x|\leq 2|y|$ for any $x\in B_R$, $y\in\bbR^n\setminus B_{2R}$, we get
\begin{align*}
    I&\leq c\left(R^{sp}\dashint_{B_R}\int_{\bbR^n\setminus B_{2R}}\frac{|u(y)-u(x)|^p}{|y|^{n+sp}}\,dy\,dx+cR^{sp}\dashint_{B_{2R}}\int_{B_{2R}}|d_su|^p\frac{\,dx\,dy}{|x-y|^n}\right)\\
    &\leq cR^{sp-n}\widetilde{\mathcal{E}}(u;B_{2R})
\end{align*}
for some constant $c=c(n,s,\Lambda,p)$, where we have used 
\begin{align*}
    |u(y)-(u)_{B_R}|^p\leq \left|\dashint_{B_R}(u(y)-u(x))\,dx\right|^p\leq \dashint_{B_R}|u(y)-u(x)|^p\,dx.
\end{align*} 
Therefore, we deduce \eqref{ineq1.reltildeee} with $x_0=0$. On the other hand, we observe
\begin{align*}
    R^{sp-n}\int_{\bbR^n\setminus B_{2R}}\int_{B_R}|d_su|^p\frac{\,dx\,dy}{|x-y|^n}&\leq  c\left(\dashint_{B_R}|u-(u)_{B_{2R}}|^p+R^{sp}\int_{\bbR^n\setminus B_{2R}}\frac{|u(y)-(u)_{B_{2R}}|^p}{|y|^{n+sp}}\right)\\
    &\leq c\left(\mathcal{E}(u;B_{2R})+R^{sp}\int_{\bbR^n\setminus B_{2R}}\frac{|u(y)-(u)_{B_{2R}}|^q}{|y|^{n+sp}}\right)\\
    &\leq c\left(\mathcal{E}(u;B_{2R})+\mathcal{E}(u;B_{2R})^{\frac{q}p}\right)
\end{align*}
for some constant $c=c(n,s,\Lambda)$, where we have used the Poincar\'e inequality and the fact that $|u|\equiv 1$ on $\bbR^n$. Thus, we have 
\begin{align*}
    R^{sp-n}\widetilde{\mathcal{E}}(u;B_R)&\leq 2R^{sp-n}\left(\int_{\bbR^n\setminus B_{2R}}\int_{B_R}|d_su|^p\frac{\,dx\,dy}{|x-y|^n}+R^{sp-n}\int_{ B_{2R}}\int_{B_{2R}}|d_su|^p\frac{\,dx\,dy}{|x-y|^n}\right)\\
    &\leq c\left(\mathcal{E}(u;B_{2R})+\mathcal{E}(u;B_{2R})^{\frac{q}p}\right),
\end{align*}
which implies \eqref{ineq2.reltildeee} with $x_0=0$. This completes the proof.
\end{proof}

In addition, we observe a basic inequality of the functional $\mathcal{E}(\cdot)$.
\begin{lemma}\label{lem.consttail}
    There is constant $c=c(n,s,q,p)$ such that
\begin{align}\label{const.tail}
    \mathcal{E}(u;B_{\rho R}(x_0))\leq c\rho^{sp-n}\mathcal{E}(u;B_{R}(x_0))
\end{align}
hold for any $\rho\in(0,1]$.
\end{lemma}
\begin{proof}
We may assume $x_0=0$. Since if $\rho\geq1/32$, \eqref{const.tail} follows by choosing the constant $c$ sufficiently large, we take $\rho<1/32$.
First, we observe
\begin{align}\label{ineq1.ctail}
    &(\rho R)^{sp}\dashint_{B_{\rho R}}\int_{B_{\rho R}}|d_su|^p\frac{\,dx\,dy}{|x-y|^{n}}\leq \rho^{sp-n} R^{sp}\dashint_{B_{ R}}\int_{B_{ R}}|d_su|^p\frac{\,dx\,dy}{|x-y|^{n}}.
\end{align}
We next see that
\begin{align*}
    &\mathrm{Tail}_q(u-(u)_{B_{\rho R}};B_{\rho R})^p\\
    &\leq c\left[\sum_{i=0}^k2^{-spi}\dashint_{B_{2^i\rho R}}|u-(u)_{B_{2^i\rho R}}|^q\,dx+2^{-spk}\mathrm{Tail}_q(u-(u)_{B_{2^k\rho R}};B_{2^k\rho R})^q\right]^{\frac{p}q}\\
    &\leq 
    c\left[\sum_{i=0}^k2^{-spi}\dashint_{B_{2^i\rho R}}|u-(u)_{B_{2^i\rho R}}|^p\,dx+2^{-spk}\mathrm{Tail}_q(u-(u)_{B_{2^k\rho R}};B_{2^k\rho R})^p\right]\\
    &\eqqcolon I_1+I_2.
\end{align*}
for some constant $c=c(n,s,q,p)$, where we choose
\begin{equation*}
    1/4\leq 2^{k}\rho< 1/2.
\end{equation*}
In addition, we have used the fact that
\begin{align*}
    \left(\sum_{i=0}^k 2^{-spi} a_i\right)^{\frac{p}q}\leq \left(\sum_{i=0}^k 2^{-spi} a_i^{\frac{p}q}\right)\left(\sum_{i=0}^{k}2^{-spi}\right)^{\frac{p-q}{q}}.
\end{align*}
We now use the Poincar\'e inequality given in Lemma \ref{lem.spi} to see that
\begin{align*}
    I_{1}&\leq c\sum_{i=0}^{k}(\rho R)^{sp}(2^{i}\rho)^{-n}\dashint_{B_{ R}}\int_{B_{ R}}|d_su|^p\frac{\,dx\,dy}{|x-y|^{n}}\leq c \rho^{sp-n}R^{sp}\dashint_{B_{ R}}\int_{B_{ R}}|d_su|^p\frac{\,dx\,dy}{|x-y|^{n}}.
\end{align*}
After a few simple computations, we deduce
\begin{align*}
    I_{2}\leq c\rho^{sp}\mathcal{E}(u;B_R(x_0)).
\end{align*}
Combining the estimate \eqref{ineq1.ctail} with the ones for $I_1$ and $I_2$ yields the desired estimate \eqref{const.tail}. 
\end{proof}

\subsection{Div-curl lemma}
An important feature of harmonic map systems (local or nonlocal) is that they exhibit some regularity by compensation. This amounts to a special algebraic structure of the nonlinearity which gives some suitable cancelations. This has been observed by many authors in the second-order case (mainly H\'elein, Evans and Rivi\`ere, see \cite{MR2431658} for references ) and has been used to derive regularity (see also \cite{sharpTopping, ChangWangYang,RiviereStruwe}).  In the nonlocal case, such a phenomenon has been exhibited primarily by Da Lio and Rivi\`ere \cite{DLR1,DLR2} and also by Mazowiecka and Schikorra \cite{MazSch18} by means of a suitbale div-curl. 
We point out that since we deal with a fractional harmonic map with a right-hand side, we need a more general version of a div-curl lemma involving a right-hand side (see Lemma \ref{lem.h1norm} for more details). To prove such a lemma, we need a rather standard higher Sobolev regularity result for solutions to a non-homogeneous fractional Laplace equation (see Lemma \ref{lem.high.diff.l1}). To prove this, first we recall a comparison estimate given by \cite{KuuMinSir15m}.
\begin{lemma}\label{lem.comp0}
    Let $u\in W^{s,2}(\bbR^n)$ be a weak solution to 
    \begin{equation}\label{eq.comp0}
\left\{
\begin{alignedat}{3}
(-\Delta)^su&= \mu&&\qquad \mbox{in  $B_{R}(x_0)$}, \\
u&=0&&\qquad  \mbox{in $\bbR^n\setminus B_{R}(x_0)$}
\end{alignedat} \right.
\end{equation}
with $\mu\in C^\infty_c(B_R(x_0))$. For a fixed $\epsilon\in\left(0,\frac{s}{2(n-s)}\right]$, we obtain
\begin{align*}
    R^{s_1}\left(\iint_{\mathcal{B}_{2R}(x_0)}|d_{s_1}u|^{1+\epsilon}{\frac{\,dx\,dy}{|x-y|^n}}\right)^{\frac1{1+\epsilon}}\leq c\left(\int_{B_R(x_0)}|R^{2s}\mu|^{1+\epsilon}\,dx\right)^{\frac{1}{1+\epsilon}}
\end{align*}
for some constant $c=c(n,s)$, where 
\begin{equation*}
    s_1\coloneqq s-\frac{n\epsilon}{2(1+\epsilon)}
\end{equation*}
and we write $\mathcal{B}_{2R}(x_0)\coloneqq B_{2R}(x_0)\times B_{2R}(x_0)$. 
\end{lemma}
\begin{proof}
    By the scaling argument, we may assume $R=1$ and $x_0=0$. By following the same lines as in the proof of \cite[Lemma 3.1 and Lemma 3.2]{KuuMinSir15m} with $v=0$ and $w=u$, we get 
    \begin{align}\label{ineq1.comp0}
    R^{\frac{s}2}\left(\iint_{\mathcal{B}_{2R}(x_0)}|d_{s_1}u|^{1+\epsilon}{\frac{\,dx\,dy}{|x-y|^n}}\right)^{\frac1{1+\epsilon}}\leq c\int_{B_R(x_0)}|R^{2s}\mu|\,dx
\end{align}
for some constant $c=c(n,s)$, as $s_1<s$ and $1+\epsilon<\frac{n}{n-s}$. By applying H\"older's inequality into the right-hand side of \eqref{ineq1.comp0}, we get the desired estimate.
\end{proof}

We are now ready to prove an auxiliary higher Sobolev estimate for non-homogeneous fractional Laplace equations. Note that our argument is based on \cite[Lemma 4.6]{DieKimLeeNow24}.
\begin{lemma}\label{lem.high.diff.l1}
Let us fix a small constant $\epsilon\in\left(0,\frac{s}{2(n-s)}\right]$. Let $u$ be a weak solution to
 \begin{equation}\label{eq.sola}
\left\{
\begin{alignedat}{3}
(-\Delta)^su&= \mu&&\qquad \mbox{in  $B_{R}(x_0)$}, \\
u&=0&&\qquad  \mbox{in $\bbR^n\setminus B_{R}(x_0)$}
\end{alignedat} \right.
\end{equation}
with $\mu \in C_c^\infty(B_R(x_0))$. Then for any $\delta\in(0,2s)$,  we have 
\begin{align}\label{est.diffl1}
R^{2s-\delta}[u]_{W^{2s-\delta,1+\epsilon}(B_{R/2}(x_0))}\leq cR^{2s}\|\mu\|_{L^{1+\epsilon}(B_R(x_0))}
\end{align}
for some constant $c=c(n,s,\epsilon,\delta)$.
\end{lemma}
\begin{proof}
We may assume $R=1$ and $x_0=0$. Moreover, for notational convenience, we write
\begin{align*}
    E_{\mathrm{loc}}(u;B_r(x_0))\coloneqq \dashint_{B_r(x_0)}|u-(u)_{B_r(x_0)}|\,dx,\\ E(u;B_r(x_0))\coloneqq E_{\mathrm{loc}}(u;B_r(x_0))+\mathrm{Tail}_1(u;B_r(x_0)).
\end{align*}
First, we consider the case $s\leq 1/2$. Then by \cite{DieNow23}, we have 
\begin{align*}
    [u]_{W^{2s-\delta,1+\epsilon}(B_{1/2})}\leq c(E(u;B_{1})+\|\mu\|_{L^{1+\epsilon}(B_{1})}).
\end{align*}
By the Sobolev-Poincar\'e inequality, the fact that $u\equiv 0$ on $B_1$ and Lemma \ref{lem.comp0}, we derive 
\begin{align}\label{ineq00.high.diff.l1}
    [u]_{W^{2s-\delta,1+\epsilon}(B_{1/2})}\leq c\|\mu\|_{L^{1+\epsilon}(B_{1})}
\end{align}
for some constant $c=c(n,s,\delta,\epsilon)$.

We now consider the case $2s-\delta>1$.
Fix $|h|\leq1/1000$ and
\begin{equation}\label{defn.sigma0}
    \sigma=2s-\frac{\delta}{2}\quad\text{and}\quad \beta=\frac{\sigma}{2s}.
\end{equation} By \cite[Lemma 2.11]{DieKimLeeNow24}, there is a covering $\{B_{|h|^{\beta}}(z_{i})\}_{i\in I}$ of $B_{1/2}$, such that $z_{i}\in B_{1/2}$, $|I||h|^{n\beta}\leq c$ and
\begin{align}
\label{el : cov.ineq0}
\sup_{x \in \mathbb{R}^n} \sum_{i\in I}\bfchi_{B_{2^{k}|h|^{\beta}}(z_{i})}(x) \leq c2^{nk}
\end{align}
for some constant $c=c(n)$, where we denote $|I|$ the number of elements in the index set $I$. We now fix a positive integer $m_{0}$ such that 
\begin{equation}
\label{el : choi.m0}
1/8\leq2^{m_{0}+4}|h|^{\beta}<1/4.
\end{equation}
By standard variational methods, there is a weak solution $w_{i}\in W^{s,2}(B_{4|h|^{\beta}}(z_{i}))\cap L^{1}_{2s}(\bbR^{n})$ to
\begin{equation*}
\left\{
\begin{alignedat}{3}
(-\Delta)^s v_{i}&= 0&&\quad \mbox{in  $B_{4|h|^{\beta}}(z_{i})$}, \\
w_{i}&=u&&\quad  \mbox{a.e. in }\RRn\setminus B_{4|h|^{\beta}}(z_{i}).
\end{alignedat} \right.
\end{equation*}
Moreover, by the linearity of the operator, we get 
\begin{equation*}
\left\{
\begin{alignedat}{3}
(-\Delta)^s (u-v_{i})&= \mu&&\quad \mbox{in  $B_{4|h|^{\beta}}(z_{i})$}, \\
u-v_{i}&=0&&\quad  \mbox{a.e. in }\RRn\setminus B_{4|h|^{\beta}}(z_{i}).
\end{alignedat} \right.
\end{equation*}
Using Lemma \ref{lem.comp0} together with Poincar\'e's inequality, we have
\begin{equation}
\label{el : comp.ineq2}
    \dashint_{B_{4|h|^{\beta}}(z_{i})}|u-v_{i}|^{1+\epsilon}\,dx\leq c|h|^{2s(1+\epsilon))\beta}\dashint_{B_{4|h|^\beta}(z_i)}|\mu|^{1+\epsilon}\,dx
\end{equation}
for some constant $c=c(n,s)$.
Next, note that
\begin{equation}\label{el : comp.ineq3}
\begin{aligned}
\dashint_{B_{|h|^{\beta}}(z_{i})}|\delta_{h}^{2}u|^{1+\epsilon}\,dx&\leq c\left(\dashint_{B_{|h|^{\beta}}(z_{i})}|\delta_{h}^{2}(u-v_{i})|^{1+\epsilon}\,dx+\dashint_{B_{|h|^{\beta}}(z_{i})}|\delta_{h}^{2}v_i|^{1+\epsilon}\,dx\right)\eqqcolon J_{1}+J_{2},
\end{aligned}
\end{equation}
where we write $\delta_h^2u(x)\coloneqq u(x+2h)+u(x)-2u(x+h)$.
By \eqref{el : comp.ineq2}, we obtain
\begin{align}\label{el : comp.ineq4}
    J_{1}\leq c|h|^{2s(1+\epsilon)\beta}\dashint_{B_{4|h|^\beta}(z_i)}|\mu|^{1+\epsilon}\,dx
\end{align}
for some constant $c=c(n,s,\epsilon)$. 
By following the same lines as in the proof of \cite[Lemma 4.5]{DieKimLeeNow24} with $\delta_h^2v$ replaced by $|\delta_h^2v|^{1+\epsilon}$ and the fact that $(1+\epsilon)<2$, we have 
\begin{align*}
    J_{2}^{\frac1{1+\epsilon}}\leq c|h|^{s(1-\beta)+1}{E}\left(\frac{\delta_{h}v_{i}}{|h|};B_{4|h|^{\beta}}(z_{i})\right)
\end{align*}
for some constant $c=c(n,s,\epsilon)$. Using this, we further estimate 
\begin{align}\label{el : comp.ineq5}
\begin{split}
J_{2}^{\frac1{1+\epsilon}}&\leq c|h|^{s(1-\beta)}\dashint_{B_{4|h|^{\beta}}(z_{i})}|u-v_{i}|\,dx+ c|h|^{s(1-\beta)+1}{E}\left(\frac{\delta_{h}u}{|h|};B_{4|h|^{\beta}}(z_{i})\right)\\
&\leq  c\left[|h|^{s(1-\beta)
	+2s\beta}\left(\dashint_{B_{4|h|^\beta}(z_i)}|\mu|^{1+\epsilon}\,dx\right)^{\frac1{1+\epsilon}}+|h|^{s(1-\beta)+1}{E}\left(\frac{\delta_{h}u}{|h|};B_{4|h|^{\beta}}(z_{i})\right)\right]
\end{split}
\end{align}
for some constant $c=c(n,s,\delta,\epsilon)$. By (4.30) in \cite{DieKimLeeNow24}, we deduce
\begin{align}\label{el : comp.ineq6}
\begin{split}
{E}\left(\frac{\delta_{h}u}{|h|};B_{4|h|^{\beta}}(z_{i})\right)\leq c\left[\sum_{j=0}^{m_{0}+1}2^{-2sj}E_{\mathrm{loc}}(\nabla u;B_{2^{j+4}|h|^{\beta}}(z_{i}))+2^{-2sm_{0}}E\left(\nabla u;B_{{3}/{4}}\right)\right].
\end{split}
\end{align}
Thus, combining \eqref{el : comp.ineq3}--\eqref{el : comp.ineq6} leads to
\begin{align*}
    \sum_{i\in I}\dashint_{B_{|h|^{\beta}}(z_{i})}|\delta_{h}^{2}u|^{1+\epsilon}\,dx&\leq c\sum_{i\in I}|h|^{2s(1+\epsilon)\beta}\dashint_{B_{4|h|^\beta}(z_i)}|\mu|^{1+\epsilon}\,dx\\
    &\quad+c|h|^{(s(1-\beta)+1)(1+\epsilon)}\sum_{i\in I}\left(\sum_{j=0}^{m_{0}+1}2^{-2sj}E_{\mathrm{loc}}(\nabla u;B_{2^{j+4}|h|^{\beta}}(z_{i}))\right)^{1+\epsilon}\\
    &\quad +c|h|^{(s(1-\beta)+1)(1+\epsilon)}\sum_{i\in I}2^{-2sm_{0}(1+\epsilon)}E\left(\nabla u;B_{{3}/{4}}\right)^{1+\epsilon}\\
    &\eqqcolon L_{1}+L_{2}+L_{3}.
\end{align*}
We now estimate each term $L_{1},L_{2}$ and $L_{3}$.

\textbf{Estimate of $L_{1}$.} By \eqref{el : cov.ineq0}, we get
\begin{align*}
    L_{1}\leq c|h|^{2s(1+\epsilon)\beta-n\beta}\int_{B_1}|\mu|^{1+\epsilon}\,dx.
\end{align*}

\textbf{Estimate of $L_{2}$.}
By H\"older's inequality, \cite[Lemma 2.1]{DieKimLeeNow24} and \eqref{el : cov.ineq0}, we estimate $L_2$ as
\begin{equation*}
\begin{aligned}
L_{2}&\leq c|h|^{(1+s(1-\beta))(1+\epsilon)}\sum_{j=0}^{m_{0}}2^{-2sj}\sum_{i\in I}\dashint_{B_{2^{j+4}|h|^{\beta}}(z_{i})}\left|\nabla u-\left(\nabla u\right)_{B_{{3}/{4}}}\right|^{1+\epsilon}\,dx\\
&\leq c|h|^{(1+s(1-\beta))(1+\epsilon)}\sum_{j=0}^{m_{0}}2^{(-2s+n)j}\sum_{i\in I}\int_{B_{3/4}}\left|\nabla u-\left(\nabla u\right)_{B_{{3}/{4}}}\right|^{1+\epsilon}\bfchi_{B_{2^{j+4}|h|^{\beta}}(z_{i})}\,dx\\
&\leq c|h|^{(1+s(1-\beta))(1+\epsilon)-n\beta}\sum_{j=0}^{m_{0}}2^{-2sj}\int_{B_{{3}/{4}}}\left|\nabla u-\left(\nabla u\right)_{B_{{3}/{4}}}\right|^{1+\epsilon}\,dx\\
&\leq c|h|^{(1+s(1-\beta))(1+\epsilon)-n\beta}\int_{B_{{3}/{4}}}\left|\nabla u-\left(\nabla u\right)_{B_{{3}/{4}}}\right|^{1+\epsilon}\,dx
\end{aligned}
\end{equation*}
for some constant $c=c(n,s,\delta,\epsilon)$.

\textbf{Estimate of $L_{3}$.}
By following the same lines as in the estimate of $L_3$ given in \cite[Lemma 4.6]{DieKimLeeNow24}, we deduce
\begin{align*}
    L_{3}\leq c|h|^{(s(1-\beta)+1+2\beta)(1+\epsilon)-n\beta}E\left(\nabla u;B_{{3}/{4}}\right)^{1+\epsilon}.
\end{align*}
By combining all the estimates, we have
\begin{align}\label{el : comp.ineq8}
\begin{split}
\int_{B_{{1}/{2}}}|\delta_{h}^{2}u|^{1+\epsilon}\,dx&\leq\sum_{i\in I}\int_{B_{|h|^{\beta}}(z_{i})}|\delta_{h}^{2}u|^{1+\epsilon}\,dx\\
&\leq c|h|^{2s\beta(1+\epsilon)}\int_{B_1}|\mu|^{1+\epsilon}\\
&\quad+c|h|^{(1+s(1-\beta))(1+\epsilon)}\int_{B_{3/4}}|\nabla u-(\nabla u)_{B_{3/4}}|^{1+\epsilon}\,dx\\
&\quad+c|h|^{(1+2s\beta+s(1-\beta))(1+\epsilon)}E\left(\nabla u;B_{{3}/{4}}\right)^{1+\epsilon}
\end{split}
\end{align}
for some constant $c=c(n,s,\delta,\epsilon)$. Now, by applying the same bootstrap argument given in the proof of \cite[Lemma 4.6]{DieKimLeeNow24} into \eqref{el : comp.ineq8}, we obtain 
\begin{align*}
    [\nabla u]_{W^{2s-\delta-1,1+\epsilon}(B_{1/2})}\leq c\left(E(u;B_1)+\left(\int_{B_1}|\mu|^{1+\epsilon}\,dx\right)^{\frac1{1+\epsilon}}\right)
\end{align*}
for some constnat $c=c(n,s,\delta,\epsilon)$ As in \eqref{ineq00.high.diff.l1}, we obtain
\begin{align*}
    [\nabla u]_{W^{2s-\delta-1,1+\epsilon}(B_{1/2})}\leq c\left(\int_{B_1}|\mu|^{1+\epsilon}\,dx\right)^{\frac1{1+\epsilon}},
\end{align*}
which completes the proof.
\end{proof}
We now explain how to obtain the existence of a distributional solution to the fractional Poisson equation satisfying the estimate \eqref{est.diffl1} when the right-hand side is in $L^{1+\epsilon}$ via approximation.
\begin{remark}\label{rmk.sola}
    Using the notion of SOLA (see \cite{KuuMinSir15m}) and Lemma \ref{lem.high.diff.l1}, we can prove that there is a distributional solution $u$ to
    \begin{equation*}
\left\{
\begin{alignedat}{3}
(-\Delta)^su&= \mu&&\qquad \mbox{in  $B_{1}$}, \\
u&=0&&\qquad  \mbox{in $\bbR^n\setminus B_{1}$}
\end{alignedat} \right.
\end{equation*}
with $\mu \in L^{1+\epsilon}(B_1)$ such that for any $\delta\in(0,2s)$, 
\begin{align*}
[u]_{W^{2s-\delta,1+\epsilon}(B_{3/4})}\leq c\|\mu\|_{L^{1+\epsilon}(B_1)}.
\end{align*}
Here, the distributional solution $u$ satisfies for any $\psi\in C_c^\infty(B_1)$,
\begin{align}\label{Defn.distr}
    \iint_{\bbR^{2n}}d_sud_s\psi\frac{\,dx\,dy}{|x-y|^n}=\int_{B_1}\mu\psi\,dx.
\end{align}
By \cite[Theorem 1.1]{KuuMinSir15m}, there are sequences of $u_k\in W^{s,2}(\bbR^n)$ and $\mu_k\in C_c^\infty(B_1)$ such that
 \begin{equation}\label{eq.sola}
\left\{
\begin{alignedat}{3}
(-\Delta)^su_k&= \mu_k&&\qquad \mbox{in  $B_{1}$}, \\
u_k&=0&&\qquad  \mbox{in $\bbR^n\setminus B_{1}$}
\end{alignedat} \right.
\end{equation}
and the function $u_k\to u$ a.e. in $\bbR^n$ and 
\begin{align*}
    \lim_{j\to\infty}\int_{B_1}|\mu_j|^{1+\epsilon}\leq \int_{B_1}|\mu|^{1+\epsilon}. 
\end{align*}
Using this and Lemma \ref{lem.high.diff.l1}, we have 
\begin{align*}
    [u_k]^{1+\epsilon}_{W^{2s-\sigma,1+\epsilon}(B_{3/4})}\leq c\int_{B_1}|\mu|^{1+\epsilon}.
\end{align*}
By Fatou's lemma together with the fact that $u_k\to u$ a.e. in $\bbR^n$, we get 
\begin{align*}
    [u]^{1+\epsilon}_{W^{2s-\sigma,1+\epsilon}(B_{3/4})}\leq c\int_{B_1}|\mu|^{1+\epsilon}.
\end{align*}
In addition, given $\epsilon>0$, we choose 
\begin{equation}
    \sigma=\frac{n\epsilon}{2(1+\epsilon)}
\end{equation} to see that 
\begin{align*}
    (2s-\sigma)-\frac{n}{1+\epsilon}\geq s-\frac{n(2+\epsilon)}{2+2\epsilon}.
\end{align*}
Since $p=\frac{n}s$, there is a sufficiently small $\varsigma=\varsigma(n,s,\epsilon)$ such that 
\begin{align}\label{varsigma.ee}
    \frac{1}{p'+\varsigma}\left(1-\frac{\varsigma}{p'}\right)>s-\frac{n(2+\epsilon)}{2+2\epsilon}.
\end{align}
Therefore, by the embedding result given e.g.\ in \cite[Proposition 2.4]{Now23v}, we get 
\begin{equation}\label{ineq.rmk.sola}
\begin{aligned}
    R^{s}[u]_{W^{s,p'}(B_{3/4})}\leq c[u]_{W^{s+\frac{n\varsigma}{p'(p'+\varsigma)},p'+\varsigma}(B_{3/4})}&\leq c[u]_{W^{2s-\sigma,1+\epsilon}(B_{3/4})}\\
    &\leq c\left(\int_{B_1}|\mu|^{1+\epsilon}\right)^{\frac1{1+\epsilon}}.
\end{aligned}
\end{equation}
\end{remark}

We are now ready to prove the following div-curl lemma with a right-hand side.
\begin{lemma}\label{lem.h1norm}
Fix $p_0\in \left(1,1+\frac{s}{2(n-s)}\right]$ and $p=\frac{n}{s}$. Let us assume that $F\in L^{p'}_{\mathrm{od}}(B_1)$ satisfies 
    \begin{equation*}
        -\divergence_s F=g\quad\text{in }B_1,
    \end{equation*}
    for some $g\in L^{p_0}(B_1)$. 
    Then for any $\psi\in C_c^\infty(B_{3/4})$, we have
    \begin{equation}\label{goal.h1norm}
    \begin{aligned}
        \left|\int_{\bbR^n}(F\odot d_s w)\psi\,dx\right|&\leq c\left(\|F\|_{L^{p'}_{\mathrm{od}}(B_1)}+\|g\|_{L^{p_0}(B_1)}\right){\widetilde{\mathcal{E}}(w;B_1)}^{\frac1p}\\
        &\qquad\times\left[[\psi]_{BMO(\bbR^n)}+\|\psi\|_{L^1(\bbR^n)}\right]
    \end{aligned}
    \end{equation}
    for some constant $c=c(n,s,p_0)$, where the dot product $\odot $  is defined in \eqref{defn.odot}.
\end{lemma}
\begin{proof}
    We may assume $g\in L^{p_0}(\bbR^n)$ by extending $g\equiv 0$ on $\bbR^n\setminus B_1$. Then by Remark \ref{rmk.sola}, there is a distributional solution $v$ to
    \begin{equation}\label{eq1.h1norm}
        (-\Delta)^s v=g\quad\text{in }B_2
    \end{equation}
    with $v\equiv 0$ on $\bbR^n\setminus B_2$ satisfying the mentioned Sobolev estimates.
    By \eqref{ineq.rmk.sola}, we get 
    \begin{equation}\label{ineq000.h1norm}
    \begin{aligned}
    \|d_sv\|_{L^{p'}_{\mathrm{od}}(B_1)}\leq [v]_{W^{s,p'}(B_{3/2})}+\int_{\bbR^n\setminus B_{3/2}}\int_{B_1}|d_sv|^{p'}\frac{\,dx\,dy}{|x-y|^n}&\leq c\left(\|g_0\|_{L^{p_0}(B_2)}+\int_{B_2}|v|^{p'}\,dx\right)\\
    &\leq c\|g_0\|_{L^{p_0}(B_2)}
    \end{aligned}
    \end{equation}
    where $c=c(n,s,p_0)$ and we have used Lemma \ref{lem.comp0} together with the Soboev-Poincar\'e inequality, as $s_1-\frac{n}{1+\epsilon}>\frac{n}{p'}$. 
Using the function $v$, we get 
    \begin{align*}
        \left|\int_{\bbR^n}(F\odot d_s w)\psi\,dx\right|&\leq \left|\int_{\bbR^n}((F-d_sv)\odot d_s w)\psi\,dx\right|+\left|\int_{\bbR^n}(d_sv\odot d_s w)\psi\,dx\right|\coloneqq I+J.
    \end{align*}
 To estimate the term $I$, we first note that
    \begin{align*}
        -\divergence_s(F-d_sv)=0\quad\text{in }B_1,
    \end{align*}
 which follows from \eqref{Defn.distr}.
 Next, by following the same lines as in the proof of \cite[Proposition 2.4]{MazSch18} together with suitable choices of bump functions, we estimate $I$ as
   \begin{equation}\label{ineq31.h1norm}
    \begin{aligned}
        I&\leq c\|F-d_sv\|_{L^{p'}_{\mathrm{od}}(B_1)}{\widetilde{\mathcal{E}}(w;B_1)}^{\frac1p}\left[[\psi]_{BMO(\bbR^n)}+\|\psi\|_{L^1(\bbR^n)}\right]\\
        &\leq c\left(\|F\|_{L^{p'}_{\mathrm{od}}(B_1)}+\|g\|_{L^{p_0}(B_1)}\right){\widetilde{\mathcal{E}}(w;B_1)}^{\frac12}\left[[\psi]_{BMO(\bbR^n)}+\|\psi\|_{L^1(\bbR^n)}\right],
    \end{aligned}
   \end{equation}
    where we have used \eqref{ineq000.h1norm} for the last inequality.

    We now use \eqref{ineq.rmk.sola} to estimate the term $J$. First, we observe
    \begin{align*}
        J&\leq \left|\int_{B_1}\int_{B_1}d_sv d_sw\frac{\,dy}{|x-y|^n}\psi(x)\,dx\right|+\left|\int_{B_{3/4}}\int_{\bbR^n\setminus B_1}d_svd_sw\frac{\,dy}{|x-y|^n}\psi(x)\,dx\right|\eqqcolon J_1+J_2,
    \end{align*}
    where we have used the fact that $\psi\equiv 0$ on $\bbR^n\setminus B_{3/4}$. We now use H\"older's inequality, \eqref{ineq.rmk.sola} with $\epsilon=p_0-1$ and \eqref{prop.bmo} to derive 
    \begin{equation}\label{ineq4.h1norm}
    \begin{aligned}
        J_1&\leq \left(\int_{B_1}\int_{B_1}\left(\frac{|d_sv|}{|x-y|^{n/p'}}\right)^{p'+\varsigma}\,dx\,dy\right)^{\frac{1}{2+\varsigma}}\left(\int_{B_1}\int_{B_1}\left(\frac{|d_sw|}{|x-y|^{n/p}}\right)^{p}\,dx\,dy\right)^{\frac{1}{p}}\\
        &\quad\times \left(\int_{B_1}\int_{B_1}\left(|\psi(x)|\right)^{{q_0}}\,dx\,dy\right)^{\frac{1}{q_0}}\\
        &\leq c[v]_{W^{s+\frac{n\varsigma}{p'(p'+\varsigma)},p'+\varsigma}(B_{1})}[w]_{W^{s,p}(B_1)}\left[\left(\dashint_{B_{\frac34}}|\psi-(\psi)_{B_{\frac34}}|^{q_0}\,dx\right)^{\frac{1}{q_0}}+(\psi)_{B_{\frac34}}\right]\\
        &\leq c\|g\|_{L^{p_0}(B_1)}{\widetilde{\mathcal{E}}(w;B_1)}^{\frac1p}\left[[\psi]_{BMO(\bbR^n)}+\|\psi\|_{L^1(\bbR^n)}\right]
    \end{aligned}
    \end{equation}
    for some constant $c=c(n,s,p_0)$, where $\frac1{q_0}=1-\frac1p-\frac{1}{p'+\varsigma}$ and the constant $\varsigma=\varsigma(n,s,p_0)\in(0,1)$ is determined in \eqref{varsigma.ee}.  In light of the fact that $|y|\leq 16|y-x|$ for any $x\in B_{3/4}$ and $y\in\bbR^n\setminus B_1$, we get
    \begin{align*}
        J_2\leq c\int_{B_{3/4}}\int_{\bbR^n\setminus B_1}|v(x)-v(y)||w(x)-w(y)||\psi(x)|\frac{\,dy}{|y|^{n+sp}}\,dx
    \end{align*}
    for some constant $c$.
    We further estimate $J_2$ as 
    \begin{align*}
        J_2&\leq c\int_{B_{3/4}}\int_{\bbR^n\setminus B_1}|v(x)-(v)_{B_{3/4}}||w(x)-(w)_{B_{3/4}}||\psi(x)|\frac{\,dy}{|y|^{n+sp}}\,dx\\
        &\quad+c\int_{B_{3/4}}\int_{\bbR^n\setminus B_1}|v(x)-(v)_{B_{3/4}}||w(y)-(w)_{B_{3/4}}||\psi(x)|\frac{\,dy}{|y|^{n+sp}}\,dx\\
        &\quad+c\int_{B_{3/4}}\int_{\bbR^n\setminus B_1}|v(y)-(v)_{B_{3/4}}||w(x)-(w)_{B_{3/4}}||\psi(x)|\frac{\,dy}{|y|^{n+sp}}\,dx\\
        &\quad+c\int_{B_{3/4}}\int_{\bbR^n\setminus B_1}|v(y)-(v)_{B_{3/4}}||w(y)-(w)_{B_{3/4}}||\psi(x)|\frac{\,dy}{|y|^{n+sp}}\,dx\eqqcolon \sum_{i=1}^4J_{2,i}
    \end{align*}
    Using H\"older's inequality and Lemma \ref{lem.spi}, we estimate $I_{2,1}$ as
    \begin{align*}
        J_{2,1}&\leq c\|v-(v)_{B_{3/4}}\|_{L^{p'}(B_{3/4})}\|w-(w)_{B_{3/4}}\|_{L^{\frac{np}{n-s}}(B_{3/4})}\|\psi\|_{L^{q_1}(B_{3/4})}\\
        &\leq c[v]_{W^{s,p'}(B_{3/4})}[w]_{W^{s,p}(B_{3/4})}\left[[\psi]_{BMO(\bbR^n)}+\|\psi\|_{L^1(\bbR^n)}\right],
    \end{align*}
    where we have used the same technique as in \eqref{ineq4.h1norm} to control the term $\|\psi\|_{L^{q_0}}$ and 
    \begin{equation}\label{const.q1}
        \frac1q_1\coloneqq1-\frac1{p'}-\frac{n-s}{np}>0.
    \end{equation}
     We next estimate $J_{2,2}$ as 
    \begin{align*}
        J_{2,2}&\leq \mathrm{Tail}_{p}(w-(w)_{B_{3/4}})\|v-(v)_{B_{3/4}}\|_{L^{p'}(B_{3/4})}\|\psi\|_{L^{p}(B_{3/4})}\\
        &\leq c{{\mathcal{E}}(w;B_{1})}^{\frac1p}[v]_{W^{s,p'}(B_{3/4})}\left[[\psi]_{BMO(\bbR^n)}+\|\psi\|_{L^1(\bbR^n)}\right],
    \end{align*}
    where we have used H\"older's inequality and Lemma \ref{lem.spi}. In addition, we have also used the fact that 
    \begin{align*}
        \mathrm{Tail}_{p}(w-(w)_{B_{3/4}};B_{3/4})&=\int_{\bbR^n\setminus B_{3/4}}\frac{|w(y)-(w)_{B_{3/4}}|^p}{|y|^{n+sp}}\,dy\\
        &\leq \int_{\bbR^n\setminus B_{3/4}}\int_{B_{3/4}}\frac{|w(y)-w(x)|^p}{|y|^{n+sp}}\,dx\,dy\leq \widetilde{\mathcal{E}}(w;B_1).
    \end{align*}Similarly, we estimate $J_{2,3}+J_{2,4}$ as 
    \begin{align*}
        J_{2,3}+J_{2,4}\leq  c[v]_{W^{s,p'}(B_{3/4})}{\widetilde{\mathcal{E}}(w;B_{1})}^{\frac1p}\left[[\psi]_{BMO(\bbR^n)}+\|\psi\|_{L^1(\bbR^n)}\right].
    \end{align*}
By combining all the estimates $J_{2,i}$ for any $i=1,2,3,4$, we obtain 
\begin{equation}\label{ineq5.h1norm}
\begin{aligned}
    J_2&\leq c\|g\|_{L^{p_0}(B_1)}{\widetilde{\mathcal{E}}(w;B_{1})}^{\frac1p}\left[[\psi]_{BMO(\bbR^n)}+\|\psi\|_{L^1(\bbR^n)}\right]
\end{aligned}
\end{equation}
for some constant $c=c(n,s,p_0)$. We now employ the estimates $J_1$ and $J_2$ given in \eqref{ineq4.h1norm} and \eqref{ineq5.h1norm}, respectively, together with \eqref{ineq31.h1norm} to derive 
\begin{align*}
    \left|\int_{\bbR^n}(F\odot d_s w)\psi\,dx\right|
    &\leq\left|\int_{\bbR^n}((F-d_sv)\odot d_s w)\psi\,dx\right|+\left|\int_{\bbR^n}(d_sv\odot d_s w)\psi\,dx\right|\\
    &\leq c\left(\|F\|_{L^{p'}_{\mathrm{od}}(B_1)}+\|g\|_{L^{p_0}(B_1)}\right){\widetilde{\mathcal{E}}(w;B_{1})}^{\frac1p}\left[[\psi]_{BMO(\bbR^n)}+\|\psi\|_{L^1(\bbR^n)}\right]
\end{align*}
for some constant $c=c(n,s,p_0)$, which completes the proof.
\end{proof}
\begin{remark}\label{rmk.h1norm}
    Assume $p=2$ and $n>2s$, then $p_0=\min\left\{\frac{n}{2s},\frac{2n}{n+s}\right\}>\frac{2n}{n+2s}$. Therefore, there is the unique weak solution $v\in \widetilde{H}^s_{00}(B_{2})$ to 
    \begin{align*}
        (-\Delta)^sv=g\quad\text{in }B_{2}
    \end{align*}
    and the solution $v$ satisfies \eqref{ineq000.h1norm} with $p'=2$. Moreover, by the self-improving property given in \cite{KuuMinSir15}, there is a small constant $\varsigma=\varsigma(n,s)$ such that
    \begin{align*}
        [v]_{W^{s+\frac{n\varsigma}{2(2+\varsigma)},2+\varsigma}(B_1)}\leq c\|g\|_{L^{p_0}(B_2)}
    \end{align*}
    for some constant $c=c(n,s)$. Thus,when $p=2$, $n>2s$ and $p_0=\min\left\{\frac{n}{2s},\frac{2n}{n+s}\right\}$, \eqref{goal.h1norm} holds.
\end{remark}

\subsection{Decay estimates for homogeneous systems}
We end this section with providing decay estimates for the functional $\widetilde{\mathcal{E}}(v;\cdot)$, where $v$ is a weak solution to a fractional p-Laplace type system.
\begin{lemma}\label{lem.fracdecay}
   Let $p=\frac{n}{s}$ and let $v\in \widehat{W}^{s,p}(B_1;\bbR^N)$ be a weak solution to 
    \begin{align*}
        \mathcal{L}v=0\quad\text{in }B_1.
    \end{align*}
    Then there is a small constant $\alpha=\alpha(n,s,\Lambda,N)$ such that $u\in C^{\alpha}(B_{1/2};\bbR^N)$. In addition, for any $\beta\leq \alpha$ with $\beta q<sp$, we get
    \begin{align}\label{goal.fracdecay}
        \mathcal{E}(v;B_\tau)\leq c\tau^{p\beta}\mathcal{E}(v;B_{1/2})
    \end{align}
    for some constant $c=c(n,s,\Lambda,N)$.
\end{lemma}
\begin{proof}
    First, we want to prove that there is a constant $\beta=\beta(n,s,\Lambda,N)$ such that $\beta q<sp$ and
    \begin{align}\label{ineq0.fracdecay}
       \|v-(v)_{B_{1/4}}\|_{C^{\beta}(B_{1/4})}\leq cE_{p}(v;B_{3/8}),
    \end{align}
    where we write 
    \begin{align*}
        E_p(v;B_{r}(z))&\coloneqq\left(\dashint_{B_{r}(z)}|v-(v)_{B_{r}(z)}|^{p}\,dx\right)^{\frac1{p}}+\mathrm{Tail}_{p-1}(v-(v)_{B_{r}(z)};B_{r}(z)).
    \end{align*}
    Let us choose a cut-off function $\xi\in C_c^\infty(B_{4R}(x_0))$ with $B_{4R}(x_0)\Subset B_1$ and $\xi\equiv 1$ on $B_{3R}(x_0)$. Then by \cite[Lemma 2.6]{DieKimLeeNow24d}, $w=u\xi$ is a weak solution to
    \begin{align*}
        \mathcal{L}w=g\quad\text{in }B_{2R}(x_0),
    \end{align*}
    where
    \begin{align}\label{ineq1.fracdecay}
        \dashint_{B_{5R/2}(x_0)}|g|^{p'}\,dx\leq cR^{-spp'}\left(\dashint_{B_{5R/2}(x_0)}|v|^p\,dx+\mathrm{Tail}(v;B_{3R}(x_0))\right)
    \end{align}
    for some constant $c=c(n,s,\Lambda,N)$.
    Thus, by following the same lines as in the proof of \cite[Theorem 1.6]{Vin25} with $|x-y|^{-(n+sp)}$ replaced by $a(x,y)|x-y|^{-(n+sp)}$, we have 
    \begin{align*}
        \|w\|_{L^\infty(B_{R}(x_0))}+[R^{\alpha}w]_{C^{\alpha}(B_{R}(x_0))}&\leq cR^{s}\left(\int_{\bbR^n}\dashint_{B_{2R}(x_0)}|d_sw|^p\frac{\,dx\,dy}{|x-y|^n}\right)^{\frac1p}\\
        &\quad+c\left( \dashint_{B_{5R/2}(x_0)}|R^{sp}g|^{p'}\,dx\right)^{\frac1{p-1}},
    \end{align*}
    where $c=c(n,s,\Lambda,N)$. We now use the fact that $w\equiv 0$ on $B_{4R}(x_0)$ and \eqref{ineq1.fracdecay}, to derive 
    \begin{align*}
        \|v\|_{L^\infty(B_R(x_0))}+[R^{\alpha}v]_{C^\alpha(B_R(x_0))}\leq c\left(\dashint_{B_{4R}(x_0)}|v|^p\,dx+\mathrm{Tail}_{p-1}(v;B_{4R}(x_0))\right)
    \end{align*}
    for any $B_{4R}(x_0)\Subset B_1$. By a standard covering argument together with the fact that $v-(v)_{B_{1/4}}$ is a weak solution to $\mathcal{L}(v-(v)_{B_{1/4}})=0$, we obtain \eqref{ineq0.fracdecay} for any $\beta\leq\alpha$.

    We are now ready to prove \eqref{goal.fracdecay}. To this end, we assume $\beta q<sp$ and we note from \eqref{ineq0.fracdecay} that
\begin{equation}\label{ineq2.fracdecay} 
\begin{aligned}
    \mathcal{E}(v;B_\tau)&\leq c\left[\tau^{sp}\dashint_{B_\tau}\int_{B_\tau}[ v]_{C^\beta(B_{1/4})}^p\frac{\,dx\,dy}{|x-y|^{n-\beta p+sp}}+\mathrm{Tail}_q(v;B_\tau)^p\right]\\
    &\leq c\left(\tau^{p\beta}E_p(v;B_{3/8})^p+{\underbrace{\mathrm{Tail}_q(v;B_\tau)^p}_{\eqqcolon I}}\right)
\end{aligned}
\end{equation}
for some constant $c=c(n,s,\Lambda,N,q)$. After a few simple computations, we get
\begin{align*}
    I\leq c\left[\sum_{k=0}^{j}2^{-spk}\dashint_{B_{2^k\tau}}|v-(v)_{B_{2^k\tau}}|^q\,dx+2^{-spj}\int_{\bbR^n\setminus B_{2^j\tau}}\frac{|v-(v)_{B_{2^j\tau}}|^q}{|y|^{n+sp}}\,dy\right]^{\frac{p}q}
\end{align*}
for some constant $c=c(n,s,N)$, where $j$ is the nonnegative integer such that
\begin{equation*}
    1/8< 2^{j}\tau\leq 1/4.
\end{equation*}
Using \eqref{ineq0.fracdecay} , we further estimate $I$ as 
\begin{align*}
    I&\leq c\left[\sum_{k=0}^{j}2^{(-sp+\beta q)k}\tau^{\beta q}E_p(v;B_{3/8})^q+\tau^{sp}\mathrm{Tail}_q(v-(v)_{B_{3/8}};B_{3/8})^q\right]^\frac{p}{q}\\
    &\leq c\tau^{\beta p}\left(E_p(v;B_{3/8})^p+\mathrm{Tail}_q(v-(v)_{B_{3/8}};B_{3/8})^{p}\right),
\end{align*}
where we have also used the fact that $sp>\beta q$.
By plugging the estimate $I$ into \eqref{ineq2.fracdecay}, we have 
\begin{align}\label{ineq3.fracdecay}
    \mathcal{E}(v;B_\tau)\leq c\tau^{p\beta}\left(E_p(v;B_{3/8})^p+\mathrm{Tail}_q(v-(v)_{B_{3/8}};B_{3/8})^{p}\right).
\end{align}
By using Poincar\'e's inequality given in Lemma \ref{lem.spi}, H\"older's inequality and the estimate $I$ given in Lemma \ref{lem.reltildeee}, we have
\begin{equation}\label{ineq4.fracdecay}
    E_p(v;B_{3/8})^p+\mathrm{Tail}_q(v-(v)_{B_{3/8}};B_{3/8})^{p}\leq c\mathcal{E}(v;B_{1/2}).
\end{equation}
Combining the estimates \eqref{ineq3.fracdecay} and \eqref{ineq4.fracdecay} yields \eqref{goal.fracdecay}, which completes the proof.
\end{proof}

\section{Decay estimates}
In this section, we will prove $\epsilon$-regularity for certain nonlocal harmonic map systems. We use the following notation:
\begin{align}\label{defn.m20}
    [u]_{M^{p,\lambda}(B_R(x_0))}\coloneqq \sup_{B_\rho(z)\subset B_R(x_0)}\left(\rho^{-\lambda}{\mathcal{E}}(u;B_\rho(z))\right)^{\frac1p}
\end{align}
for any $\lambda\in[0,n)$, where the functional $\mathcal{E}(u;\cdot)$ is determined in \eqref{defn.cale}.

Then, we observe from the Poincar\'e inequality given in Lemma \ref{lem.spi} that
\begin{equation}\label{prop.bmosmall}
\begin{aligned}
    [u]_{BMO(B_R(x_0))}&\leq \sup_{B_\rho(z)\subset B_R(x_0)}\left(\dashint_{B_\rho(z)}|u-(u)_{B_{\rho}(z)}|^p\,dx\right)^{\frac1p}\\
    &\leq c\sup_{B_\rho(z)\subset B_R(x_0)}\rho^{(sp-n)/p}[u]_{W^{s,p}(B_\rho(z))}.
\end{aligned}
\end{equation}
We now give a simple inequality about the BMO semi-norm.
\begin{lemma}\label{lem.locbmo}
    Let $\psi\in C_c^\infty(B_{3/4})$ with $\psi\equiv 1$ on $B_{5/8}$. Then we have 
    \begin{align*}
        [(u-(u)_{B_{1}})\psi]_{BMO(\bbR^n)}\leq c\left([u]_{BMO(B_{1})}+[u]_{M^{p,0}(B_{1})}\right)
    \end{align*}
    for some constant $c=c(n,s)$.
\end{lemma}
\begin{proof}
Let us write $v\coloneqq u-(u)_{B_{1}}$. Now, we want to show that for any $B_r(z)\subset \bbR^n$,
\begin{equation}\label{ineq0.locbmo}
    \sup_{B_r(z)}\dashint_{B_r(z)}|v\psi-(v\psi)_{B_r(z)}|\,dx\leq c\left([u]_{BMO(B_1)}+[u]_{M^{p,0}(B_1)}\right)
\end{equation}
for some constant $c=c(n,s)$. Let $r>1/64$ to see that
    \begin{align}\label{ineq01.locbmo}
        \dashint_{B_r(z)}|v\psi-(v\psi)_{B_r(z)}|\leq c\int_{B_r(z)}|v\psi|\leq c\dashint_{B_1}|u-(u)_{B_1}|\leq c[u]_{BMO(B_1)}.
    \end{align}
    Suppose $r\leq 1/64$. If $B_r(z)\cap B_{7/8}=\emptyset$, then $\psi\equiv 0$ on $B_r(z)$. Therefore, we get
    \begin{align}\label{ineq.012.locbmo}
        \dashint_{B_r(z)}|v\psi-(v\psi)_{B_r(z)}|=0.
    \end{align}
    We now assume $B_r(z)\cap B_{7/8}\not= \emptyset$. Then we have 
    \begin{equation}\label{inc.locbmo}
        B_r(z)\subset B_1
    \end{equation}
    and
    \begin{align*}
        I\coloneqq\dashint_{B_{r}(z)}|v\psi-(v\psi)_{B_r(z)}|&\leq cr^{s-n/p}[v\psi]_{W^{s,p}(B_r(z))}\\
        &\leq c\left(r^{s-n/p}[v]_{W^{s,p}(B_r(z))}+cr^{1-n/p}\|v\|_{L^p(B_r(z))}\right),
    \end{align*}
    where $c=c(n,s)$.
    If $n\leq p$, we deduce from \eqref{prop.bmo} that 
    \begin{equation}\label{ineq1.locbmo}
    \begin{aligned}
        I\leq c(R^{sp-n}[u]_{W^{s,p}(B_r(z))}+\|v\|_{L^p(B_1)})&\leq c(R^{sp-n}[u]_{W^{s,p}(B_r(z))}+\|u-(u)_{B_1}\|_{L^p(B_1)})\\
        &\leq c([u]_{M^{p,0}(B_1)}+[u]_{BMO(B_1)}).
    \end{aligned}
    \end{equation}
    If $n>p$, we now use the H\"older's inequality to see that 
    \begin{align*}
        I&\leq c\left(r^{s-n/p}[v]_{W^{s,p}(B_r(z))}+\|v\|_{L^n(B_r(z))}\right)\\
        &\leq c\left(r^{s-n/p}[u]_{W^{s,p}(B_r(z))}+\|u-(u)_{B_1}\|_{L^n(B_1)}\right).
    \end{align*}
    As in the proof of \eqref{ineq1.locbmo}, we get 
    \begin{align}\label{ineq2.locbmo}
        I\leq c([u]_{M^{p,0}(B_1)}+[u]_{BMO(B_1)})
    \end{align}
    for some constant $c=c(n,s,p)$. Combining all the estimates \eqref{ineq01.locbmo}, \eqref{ineq.012.locbmo} and \eqref{ineq2.locbmo} gives \eqref{ineq0.locbmo}, which implies the desired estimate.
\end{proof}
Next, we want to prove decay estimates for the functional $\mathcal{E}(u;\cdot)$ under a smallness assumption, where $u$ is a weak solution to \eqref{eq.main}. Before proving this, we give the another version of the div-curl lemma given by Lemma \ref{lem.h1norm}  by specifying the function $F$ as in  \eqref{F.specify}, which will be useful for proving the desired decay estimates.
\begin{lemma}\label{lem.h1norm2}
  Let $p=\frac{n}{s}$ and let us fix $p_0\in\left(1,1+\frac{s}{2(n-s)}\right]$. Let us assume that 
   \begin{equation}\label{F.specify}
       F(x,y)\coloneqq |d_su|^{p-2}(d_s{u}^j(x,y)u^i(x)-d_s{u}^i(x,y)u^j(x))a(x,y)
   \end{equation}satisfy 
    \begin{equation*}
        -\divergence_s F=g\quad\text{in }B_1,
    \end{equation*}
    for some $g\in L^{p_0}(B_1)$, where $u^i\in W^{s,p}(B_1)\cap L^q_{sp}(\bbR^n)$ for each $i$. 
    Then for any $\psi\in \widehat{W}^{s,p}_{00}(B_{3/4})$ and $w\in \widehat{W}^{s,p}_{00}(B_{3/4})$, we have
    \begin{equation}\label{goal.h1norm2}
    \begin{aligned}
        \left|\int_{\bbR^n}(F\odot d_s w)\psi\,dx\right|&\leq c\left(\mathcal{E}(u;B_{1})^{\frac1{p'}}+\|g\|_{L^{p_0}(B_1)}\right){\widetilde{\mathcal{E}}(w;B_1)}^{\frac1p}\\
        &\qquad\times\left[[\psi]_{BMO(\bbR^n)}+\|\psi\|_{L^1(\bbR^n)}\right]
    \end{aligned}
    \end{equation}
    for some constant $c=c(n,s,\Lambda,p,q,p_0)$, where the dot product $\odot $  is defined in \eqref{defn.odot}.
\end{lemma}
\begin{proof}
First, we give a density argument to ensure that \eqref{goal.h1norm2} follows when \eqref{goal.h1norm2} holds for $\psi\in C^\infty_c(B_{7/8})$. To this end, note that there is a smooth function $\phi\in C^\infty_c(B_1)$ such that $\phi\geq0$ and $\int_{B_1}\phi=1$. Then we observe 
\begin{align*}
    \dashint_{B_r(z)}|\psi_{\epsilon}-(\psi_\epsilon)_{B_r(z)}|\,dx&=\dashint_{B_r(z)}\left|\dashint_{B_r(z)}(\psi_{\epsilon}(x)-\psi_\epsilon(y))\,dy\right|\,dx\\
    &=\dashint_{B_r(z)}\left|\dashint_{B_r(z)}\int_{\bbR^n}(\psi(x-\xi)-\psi(y-\xi))\phi_\epsilon(\xi)\,d\xi\,dy\right|\,dx\\
    &= \dashint_{B_r(z)}\left|\int_{\bbR^n}\dashint_{B_r(z)}(\psi(x-\xi)-\psi(y-\xi))\phi_\epsilon(\xi)\,dy\,d\xi\right|\,dx,
\end{align*}
where we write $\psi_\epsilon(x)=(\psi*\phi_\epsilon)(x)$ with $\phi_\epsilon(x)\coloneqq \epsilon^{-n}\phi(x/\epsilon)$. Moreover, we further estimate 
\begin{align*}
    \dashint_{B_r(z)}|\psi_{\epsilon}-(\psi_\epsilon)_{B_r(z)}|\,dx&\leq \int_{\bbR^n}\int_{B_r(z)}\left|\dashint_{B_r(z)}(\psi(x-\xi)-\psi(y-\xi))\,dy\right|\,dx\phi_\epsilon(\xi)\,d\xi\\
    &=\int_{\bbR^n}\int_{B_r(z+\xi)}\left|\dashint_{B_r(z+\xi)}(\psi(x)-\psi(y))\,dy\right|\,dx\phi_\epsilon(\xi)\,d\xi\\
    &\leq \int_{\bbR^n}[\psi]_{BMO(\bbR^n)}\phi_\epsilon(\xi)\,d\xi=[\psi]_{BMO(\bbR^n)}.
\end{align*}
From this, it suffices to show the estimate \eqref{goal.h1norm2} for the function $\psi\in C^\infty_c(B_{7/8})$.

    Let us choose a cut off function $\eta(x)\in C_{c}^\infty(B_{15/16})$ such that $\eta(x)\equiv 1$ on $B_{7/8}$. Let us write $\widetilde{F}(x,y)=F(x,y)\eta(x)\eta(y)$. Note that $\widetilde{F}(x,y)=-\widetilde{F}(y,x)$, as ${F}(x,y)=-{F}(y,x)$. Then we have 
    \begin{align*}
        -\divergence_s(\widetilde{F})=g+G\quad\text{in }B_{13/16},
    \end{align*}
    where 
    \begin{align*}
        G(x,y)\coloneqq 2\int_{\bbR^n\setminus B_{7/8}}\frac{\widetilde{F}(x,y)}{|x-y|^{n+s}}a(x-y)\,dy-2\int_{\bbR^n\setminus B_{7/8}}\frac{{F}(x,y)}{|x-y|^{n+s}}a(x-y)\,dy
    \end{align*}
    for any $x\in B_{13/16}$. Using this function $\widetilde{F}$, we have 
    \begin{equation}\label{estij.h1norm2}
    \begin{aligned}
        \left|\int_{\bbR^n}(F\odot d_s w)\psi\,dx\right|&\leq \left|\int_{\bbR^n}(\widetilde{F}\odot d_s w)\psi\,dx\right|+\left|\int_{\bbR^n}((\widetilde{F}-F)\odot d_s w)\psi\,dx\right|\\
        &\coloneqq I+J.
    \end{aligned}
    \end{equation}
    Then using Lemma \ref{lem.h1norm}, we estimate $I$ as  
    \begin{equation}\label{estI.h1norm2}
    \begin{aligned}
        I&\leq c\left(\|\widetilde{F}\|_{L^{p'}_{\mathrm{od}}(B_1)}+\|g\|_{L^{p_0}(B_1)}\right){\widetilde{\mathcal{E}}(w;B_1)}^{\frac1p}\left[[\psi]_{BMO(\bbR^n)}+\|\psi\|_{L^1(\bbR^n)}\right]\\
&\leq c\left(\mathcal{E}(u;B_{1})^{\frac1{p'}}+\|g\|_{L^{p_0}(B_1)}\right){\widetilde{\mathcal{E}}(w;B_1)}^{\frac1p}\left[[\psi]_{BMO(\bbR^n)}+\|\psi\|_{L^1(\bbR^n)}\right]
    \end{aligned}
    \end{equation}
    for some constant $c=c(n,s,\Lambda,q,p_0)$, as $\widetilde{F}(x,y)=0$ if $(x,y)\notin B_{1}\times B_1$.
We are going to estimate $J$ as
    \begin{align*}
        J&\leq c\int_{B_{3/4}}\int_{\bbR^n\setminus B_{7/8}}|F(x,y)|\frac{|w(x)|}{|x-y|^s}\psi(x)\frac{\,dy\,dx}{|x-y|^n}\\
        &\leq c\int_{\bbR^n\setminus B_{7/8}}\int_{B_{3/4}}\left(|u(x)-(u)_{B_1}|^{p-1}+|u(y)-(u)_{B_1}|^{p-1}\right)\frac{|w(x)|}{|y|^{n+sp}}|\psi(x)|\,dy\,dx\\
        &\leq c\int_{B_{3/4}}|u(x)-(u)_{B_1}|^{p-1}|w(x)||\psi(x)|\,dx\\
        &\quad+c\int_{\bbR^n\setminus B_{7/8}}|u(y)-(u)_{B_1}|^{p-1}\frac{\,dy}{|y|^{n+sp}}\int_{B_1}|w(x)||\psi(x)|\,dx\\
        &\eqqcolon J_1+J_2,
    \end{align*}
    where $c=c(n,s,\Lambda)$.
    By following the same lines as in the estimate of $J_{2,1}$ given in Lemma \ref{lem.h1norm} together with the fact that $w\equiv 0$ on $\bbR^n\setminus B_1$, we have 
    \begin{align*}
        J_1\leq c\mathcal{E}(u;B_{1})^{\frac1{p'}}\mathcal{E}(w;B_{1})^{\frac1p}\left[[\psi]_{BMO(\bbR^n)}+\|\psi\|_{L^1(\bbR^n)}\right].
    \end{align*}
    Also by the estimate $J_{2,2}$ given in Lemma \ref{lem.h1norm}, we deduce
    \begin{align*}
       I_2\leq c\mathcal{E}(u;B_{1})^{\frac1{p'}}\mathcal{E}(w;B_{1})^{\frac1p}\left[[\psi]_{BMO(\bbR^n)}+\|\psi\|_{L^1(\bbR^n)}\right].
    \end{align*}
    Thus, we have 
    \begin{align}\label{esti.h1norm2}
        J\leq c\mathcal{E}(u;B_{1})^{\frac1{p'}}\mathcal{E}(w;B_{1})^{\frac1p}\left[[\psi]_{BMO(\bbR^n)}+\|\psi\|_{L^1(\bbR^n)}\right].
    \end{align}
     Plugging the estimates \eqref{estI.h1norm2} and \eqref{esti.h1norm2} into \eqref{estij.h1norm2} leads to the desired result.
\end{proof}

We are now ready to prove the following decay estimates.
\begin{lemma}\label{lem.decayu} Let us fix $p=\frac{n}{s}$ and $p_0\in\left(1,1+\frac{s}{2(n-s)}\right]$. Recall the constants $q>p-1$ and $\beta<1$ determined in \eqref{cond.q} and Lemma \ref{lem.fracdecay}, respectively. Let $u\in \widehat{W}^{s,p}(B_{2R}(x_0);\bbR^N)$ with $|u|\equiv 1$ on $\bbR^n$ be a weak solution to 
\begin{equation}\label{eq1.dec}
    \mathcal{L}_{s,p}u=|\mathbf{d}_{s,a}u|^pu+\mu\quad\text{in }B_{R}(x_0).
\end{equation}
Suppose 
\begin{align}\label{cond.usmall}
    [u]_{M^{p,0}(B_{2R}(x_0))}\leq \epsilon
\end{align}
for some $\epsilon>0$. Then there is a sufficiently small $\epsilon_0=\epsilon_0(n,s,N,\Lambda,q,p_0)$ such that if $\epsilon\leq \epsilon_0$, then we have for any $\rho\in(0,1]$,
\begin{equation}\label{ineq.decayu}
\begin{aligned}
    \mathcal{E}(u;B_{\rho R}(x_0))&\leq c\left(\mathcal{A}(\rho,\epsilon)\mathcal{E}(u;B_R(x_0))+\frac{1}{\rho^{n-sp}}\left(\dashint_{B_R(x_0)}|R^{sp}\mu|^{p_0}\,dx\right)^{\frac{p}{p_0(p-1)}}\right)
\end{aligned}
\end{equation}
for some constant $c=c(n,s,\Lambda,N,q,p_0)$, where we write
\begin{align*}
    \mathcal{A}(\rho,\epsilon)\coloneqq \left[\rho^{p\beta}+\frac{\epsilon^{q-p+1}}{\rho^{n-sp}} \right].
\end{align*}
\end{lemma}
\begin{proof}
    We may assume $B_R(x_0)=B_1$. In addition, if $\rho>2^{-10}$, then \eqref{ineq.decayu} follows directly by taking the constant $c$ sufficiently large. Let us assume $\rho\in (0,2^{-10}]$.
   By standard variational methods, there is a unique weak solution $v\in \widehat{W}^{s,p}(B_{1/2};\bbR^N)$ to 
\begin{equation}\label{eq1.decayu}
\left\{
\begin{alignedat}{3}
\mathcal{L}v&= 0&&\qquad \mbox{in  $B_{1/2}$}, \\
v&=u&&\qquad  \mbox{in $\bbR^n\setminus B_{1/2}$}
\end{alignedat} \right.
\end{equation}
with $v-u\in \widetilde{W}_{00}^{s,p}(B_{1/2};\bbR^N)$.
Using the function $v$, we obtain
\begin{align*}
    \mathcal{E}(u;B_{\rho })\leq \mathcal{E}(u-v;B_{\rho })+\mathcal{E}(v;B_{\rho })\coloneqq I_1+I_2.
\end{align*}
In light of \eqref{const.tail} and \eqref{ineq1.reltildeee}, we have
\begin{align}\label{ineq1.decayu}
    I_1\leq c\rho^{sp-n}\mathcal{E}(u-v;B_{1/4 })\leq c\rho^{sp-n}\widetilde{\mathcal{E}}(u-v;B_{1/2}),
\end{align}
where $c=c(n,s,\Lambda,N)$.
By Lemma \ref{lem.fracdecay}, we next estimate $I_2$ as
\begin{align*}
    I_2\leq c\rho^{\beta p}\mathcal{E}(v;B_{1/4})\leq c\rho^{\beta p}\left[\mathcal{E}(u;B_{1/4})+\mathcal{E}(u-v;B_{1/4})\right]
\end{align*}
for some constant $c=c(n,s,\Lambda,N,,q)$. We now use \eqref{const.tail} and Lemma \ref{lem.reltildeee} to get that 
\begin{align}\label{ineq2.decayu}
    I_2\leq c\rho^{\beta p}\left[\mathcal{E}(u;B_{1})+\widetilde{\mathcal{E}}(u-v;B_{1/2})\right]
\end{align}
for some constant $c=c(n,s,\Lambda,N,q)$. By \eqref{ineq1.decayu} and \eqref{ineq2.decayu}, we deduce 
\begin{align}\label{ineq21.decayu}
   \mathcal{E}(u;B_{\rho })\leq c\left(\rho^{\beta p} \mathcal{E}(u;B_{1}) +\rho^{sp-n}{\underbrace{\widetilde{\mathcal{E}}(u-v;B_{1/2})}_{\eqqcolon L}}\right),
\end{align}
where $c=c(n,s,\Lambda,N,q)$.

We now estimate the term $L$. To this end, first let us write $w\coloneqq u-v$. First, using the fact that $w\equiv 0$ on $\bbR^n\setminus B_{1/2}$ and Lemma \ref{lem.spi}, we derive
\begin{align*}
    \widetilde{\mathcal{E}}(w;B_{1/2})&\leq \iint_{B_{1}\times B_1}|d_sw|^p\frac{\,dx\,dy}{|x-y|^n}+2\iint_{(\bbR^n\setminus B_{1/2})\times B_1}|w(x)|^p\frac{\,dx\,dy}{|x-y|^{n+sp}}\\
    &\leq \iint_{B_{1}\times B_1}|d_sw|^p\frac{\,dx\,dy}{|x-y|^n}+c\int_{ B_1}|w(x)|^p\,dx\\
    &\leq c\iint_{B_{1}\times B_1}|d_sw|^p\frac{\,dx\,dy}{|x-y|^n}.
\end{align*}
Using this together with a few simple computations, we deduce
\begin{align}\label{ineq02.decayu}
     \widetilde{\mathcal{E}}(w;B_{1/2})\eqsim \widetilde{\mathcal{E}}(w;B_{1})\eqsim \iint_{B_{1}\times B_1}|d_sw|^p\frac{\,dx\,dy}{|x-y|^n}.
\end{align}
Next based on \cite[Lemma 3.8]{MilPegSch21} and \cite[Lemma 5.1]{MazSch18}, we can rewrite the equation \eqref{eq1.dec} as
\begin{equation}\label{eq3.decayu}
    \mathcal{L}{u}^i=\left[\sum_j{\mathbf{\Omega}}^{i,j}\odot d_s{u}^j+{T}^i\right]+{\mu}^i \quad\text{in }B_{1},
\end{equation}
where 
\begin{equation}\label{defn.omegaij}
    {\mathbf{\Omega}}^{i,j}(x,y)\coloneqq |d_su|^{p-2}(d_s{u}^j(x,y)u^i(x)-d_s{u}^i(x,y)u^j(x))a(x,y)
\end{equation}
and
\begin{equation*}
    {T}^i\coloneqq \frac12\int_{\bbR^n}\frac{|{u}(x)-{u}(y)|^p}{|x-y|^{n+sp}}(u^i(x)-u^i(y))a(x,y)\,dy.
\end{equation*}
In addition, we get
\begin{align}\label{eq4.decayu}
    -\divergence_s(\mathbf{\Omega}^{i,j})=u^i\mu^j-u^j\mu^i\quad\text{in }B_{1}.
\end{align}
Now using \eqref{eq1.decayu}, \cite[Lemma 2.1]{DieKimLeeNow24d} and \eqref{ineq02.decayu}, we derive
\begin{equation}\label{estl.decayu}
\begin{aligned}
    L\leq \iint_{\mathcal{C}B_{1/2}}|d_s(u-v)|^{p}|x-y|^{-n}\,dx\,dy&\leq c\iint_{\mathcal{C}B_{1/2}}|d_s(u-v)|^{p}{a(x,y)}|x-y|^{-n}\,dx\,dy\\
    &\leq c<(\mathcal{L}u-\mathcal{L}v),w>
\end{aligned}
\end{equation}
when $p\geq2$ and 
\begin{equation}\label{estll.decayu}
\begin{aligned}
    L&\leq \iint_{\mathcal{C}B_{1/2}}|d_s(u-v)|^{p}|x-y|^{-n}\,dx\,dy\\
    &\leq \iint_{B_{1}\times B_1}|d_s(u-v)|^{p}|x-y|^{-n}\,dx\,dy\\
    &\leq \left(\iint_{B_{1}\times B_1}(|d_su|+|d_sv|)|^{p}|x-y|^{-n}\,dx\,dy\right)^{\frac{2-p}{2}}\\
    &\quad\times\left(\iint_{B_{1}\times B_1}(|d_su|+|d_sv|)|^{p-2}|d_s(u-v)|^2|x-y|^{-n}\,dx\,dy\right)^{\frac{p}{2}}\\
    &\leq c\left(\mathcal{E}(u;B_1)+\widetilde{\mathcal{E}}(w;B_1)\right)^{\frac{2-p}{2}}\left(<(\mathcal{L}u-\mathcal{L}v),w>\right)^{\frac{p}{2}}
\end{aligned}
\end{equation}
when $p\leq 2$, where we write 
\begin{align*}
    <\mathcal{L}u,w>=\iint_{\bbR^{2n}}|d_su|^{p-2}d_su d_sw \frac{a(x,y)}{|x-y|^n}.
\end{align*}Using \eqref{eq3.decayu} together with the fact that $w=u-v\in \widehat{W}^{s,p}_{00}(B_{1/2})$, we obtain
\begin{align*}
    (\mathcal{L}u-\mathcal{L}v)w=\int_{B_{1/2}}\mathbf{\Omega}\odot d_su w+\int_{B_{1/2}}T w+\int_{B_{1/2}}\mu w\eqqcolon \sum_{i=1}^3L_i,
\end{align*}
where we write
\begin{align*}
    \int_{B_{1/2}}\mathbf{\Omega}\odot d_su w=\sum_{i}\int_{B_{1/2}}\left(\sum_{j}\mathbf{\Omega}^{i,j}\odot d_su^j \right)(x)w^i(x)\,dx
\end{align*}
and
\begin{align*}
    \int_{B_{1/2}}Tw=\sum_i\int_{B_{1/2}}T^i(x)w^i(x)\,dx.
\end{align*}
To further estimate \eqref{estl.decayu} and \eqref{estll.decayu}, we estimate each term $L_1,L_2$ and $L_3$.

\textbf{Estimate of $L_1$}.
Before we estimate $L_1$, let us take a cut-off function $\psi\in C_c^\infty(B_{3/4})$ with $\psi\equiv 1$ on $B_{5/8}$. Then we observe 
\begin{equation}\label{estj.decayu}
\begin{aligned}
    J&\coloneqq\int_{B_{1/2}}\int_{\bbR^n}{\mathbf{\Omega}}^{i,j}(x,y) d_su^j(x,y)w^i(x)k(x,y)\,dy\,dx\\
    &=
    \int_{B_{1/2}}\int_{\bbR^n}{\mathbf{\Omega}}^{i,j}(x,y)d_s((u^j-(u^j)_{B_{1}})\psi)(x,y) w^i(x)k(x,y)\,dy\,dx\\
    &\quad-\int_{B_{1/2}}\int_{\bbR^n\setminus B_{5/8}}{\mathbf{\Omega}}^{i,j}(x,y)d_s((u^j-(u^j)_{B_{1}})\psi)(x,y) w^i(x)k(x,y)\,dy\,dx\\
    &\quad +\int_{B_{1/2}}\int_{\bbR^n\setminus B_{5/8}}{\mathbf{\Omega}}^{i,j}(x,y)d_su^j(x,y) w^i(x)k(x,y)\,dy\,dx\coloneqq J_1+J_2+J_3,
\end{aligned}
\end{equation}
where we write $k(x,y)\coloneqq |x-y|^{-n}$.
Let us write $\overline{u}^j\coloneqq (u^j-(u^j)_{B_{1}})\psi$.
We then obtain 
\begin{align*}
    J_1&=\int_{\bbR^n}\int_{\bbR^n}{\mathbf{\Omega}}^{i,j}(x,y)d_s\overline{u}^j(x,y) w^i(x)k(x,y)\,dy\,dx\\
    &=-\int_{{\bbR^n}}\int_{\bbR^n}\mathbf{\Omega}^{i,j}(x,y)\overline{u}^j(y) d_sw^i(x,y)k(x,y)\,dy\,dx\\
    &\quad+\int_{\bbR^n}\int_{\bbR^n}\mathbf{\Omega}^{i,j}(x,y)d_s(\overline{u}^jw^i)k(x,y)\,dy\,dx\\
    &=-\int_{{\bbR^n}}\int_{\bbR^n}\mathbf{\Omega}^{i,j}(x,y)\overline{u}^j(y) d_sw^i(x,y)k(x,y)\,dy\,dx\\
    &\quad+\int_{B_{1/2}}(u^i\mu^{j}-u^j\mu^i)(\overline{u}^jw^i)\,dx\eqqcolon J_{1,1}+J_{1,2},
\end{align*}
where we have used the fact that $d_s(fg)(x,y)=d_sf(x,y)g(x)+d_sg(x,y)f(y)$, \eqref{eq4.decayu} and $\overline{u}^jw^i\equiv0$ on $\bbR^n\setminus B_{1/2}$. We now use Lemma \ref{lem.h1norm2} to get that
\begin{align*}
    |J_{1,1}|&= \left|\int_{{\bbR^n}}(\mathbf{\Omega}^{i,j}\odot d_sw^i)(x)\overline{u}^j(x) \,dx\right|\\
    &\leq c(\mathcal{E}(u;B_{1})^{\frac1{p'}}+\|\mu\|_{L^{p_0}(B_{1})})\widetilde{\mathcal{E}}(w;B_{1})^{\frac1p}\left[\|\overline{u}^j\|_{L^1(\bbR^n)}+[\overline{u}^j]_{BMO(\bbR^n)}\right]
\end{align*}
for some constant $c=c(n,s,\Lambda,N,p_0)$.
Using this, Lemma \ref{lem.locbmo} and the Sobolev Poincar\'e inequality, we further estimate $J_{1,1}$ as
\begin{align}\label{decayu0.j11}
   |J_{1,1}| &\leq c\left({\mathcal{E}}(u;B_{1})^{\frac1{p'}}+\|\mu\|_{L^{p_0}(B_{1})}\right)\widetilde{\mathcal{E}}(w;B_{1})^{1/p}\left([u]_{M^{p,0}(B_1)}+[u]_{BMO(B_1)}\right)
\end{align}
for some constant $c=c(n,s,\Lambda,N,q,p_0)$.
We now employ H\"older's inequality, \eqref{cond.usmall} and \eqref{prop.bmosmall} to see that 
\begin{align*}
    |J_{1,1}| &\leq c\epsilon\left(\mathcal{E}(u;B_{1})+\|\mu\|_{L^{p_0}(B_1)}^{p'}+\widetilde{E}(w;B_1)\right),
\end{align*}
where $c=c(n,s,\Lambda,N,q,p_0)$.
On the other hand, in light of Young's inequality,  Poincar\'e's inequality given in Lemma \ref{lem.spi}, the fact that $|u|\equiv 1$, we estimate $J_{1,2}$ as 
\begin{equation}\label{ineqj12.decayu}
\begin{aligned}
    |J_{1,2}|&\leq \delta\widetilde{\mathcal{E}}(w;B_{1})+c\delta^{-\frac1{p-1}}\left(\int_{B_{1/2}}|\mu|^{p_0}\right)^{\frac{p'}{p_0}}
\end{aligned}
\end{equation}
for some constant $c=c(n,s,\Lambda,N,p_0)$, where the constant $\delta\in(0,1)$ will be determined later. 
Thus combining all the estimates $J_{1,1}$ and $J_{1,2}$ yields
\begin{equation}\label{estj1.decayu}
\begin{aligned}
    J_1&\leq c\left(\epsilon\mathcal{E}(u;B_{1})+(\delta+\epsilon)\widetilde{\mathcal{E}}(w;B_{1})+\delta^{-1}\|\mu\|_{L^{p_0}(B_1)}^{p'}\right),
\end{aligned}
\end{equation}
where $c=c(n,s,\Lambda,N,p_0)$.

We next estimate $J_2$ as 
\begin{align*}
    |J_2|&\leq c\int_{B_{1/2}}\int_{\bbR^n\setminus B_{5/8}}|d_s(u-(u)_{B_1})|^{p-1}|d_s\overline{u}^j||w^i(x)|k(x,y){\,dy\,dx}\\
    &\leq c\int_{B_{1/2}}\int_{\bbR^n\setminus B_{5/8}}|u(x)-(u)_{B_1}|^{p-1}|\overline{u}(x)||w(x)|\frac{\,dy\,dx}{|y|^{n+sp}}\\
    &\quad+c\int_{B_{1/2}}\int_{\bbR^n\setminus B_{5/8}}|u(x)-(u)_{B_1}|^{p-1}|\overline{u}(y)||w(x)|\frac{\,dy\,dx}{|y|^{n+sp}}\\
    &\quad+c\int_{B_{1/2}}\int_{\bbR^n\setminus B_{5/8}}|u(y)-(u)_{B_1}|^{p-1}|\overline{u}(x)||w(x)|\frac{\,dy\,dx}{|y|^{n+sp}}\\
    &\quad+c\int_{B_{1/2}}\int_{\bbR^n\setminus B_{5/8}}|u(y)-(u)_{B_1}|^{p-1}|\overline{u}(y)||w(x)|\frac{\,dy\,dx}{|y|^{n+sp}}\eqqcolon \sum_{i=1}^4J_{2,i}.
\end{align*}
First, we estimate $J_{2,1}$ as 
\begin{align*}
    |J_{2,1}|&\leq c\|u-(u)_{B_1}\|^{p-1}_{L^{p}(B_{1/2)}}\|w\|_{L^{\frac{np}{n-s}}(B_{1/2})}\|\overline{u}\|_{L^{q_1}(B_{1/2})}\\
    &\leq c[u]^{p-1}_{W^{s,p}(B_1)}[w]_{W^{s,p}(B_1)}\|u-(u)_{B_1}\|_{L^{q_1}(B_{1/2})}\\
    &\leq c\mathcal{E}(u;B_1)^{p-1}\widetilde{\mathcal{E}}(w;B_1)[u]_{BMO(B_1)}
\end{align*}
for some constant $c=c(n,s,\Lambda,N,q)$, where we have used H\"older's inequality together with the choice of $q_1$ given in \eqref{const.q1}, Lemma \ref{lem.spi} and \eqref{prop.bmo}. In addition, the constant $q_1$ is determined in \eqref{const.q1}.
Next, we estimate $J_{2,2}$ as 
\begin{align*}
    |J_{2,2}|&\leq c\left[\int_{\bbR^n\setminus B_{5/8}}\frac{|u(y)-(u)_{B_1}|}{|y|^{n+sp}}\,dy\|u-(u)_{B_1}\|^{p-1}_{L^p(B_{1/2})}\|w\|_{L^p(B_{1/2}))}\right]\\
    &\leq c\mathcal{E}(u;B_1)^{\frac1p}\mathcal{E}(u;B_{1})^{\frac1{p'}}\widetilde{\mathcal{E}}(w;B_1)^{\frac1p}\\
    &\leq c\epsilon(\mathcal{E}(u;B_1)+\widetilde{\mathcal{E}}(w;B_1)),
\end{align*}
where we have used Lemma \ref{lem.spi}, \eqref{cond.usmall} and Young's inequality.
Similarly, we estimate $J_{2,3}$ as 
\begin{align*}
    |J_{2,3}|\leq c\mathcal{E}(u;B_1)^{\frac1{p'}}\int_{B_{1/2}}|u-(u)_{B_1}||w|\,dx
    &\leq c\mathcal{E}(u;B_1)^{\frac1{p'}}\|u-(u)_{B_1}\|_{L^{\frac{np}{n-s}}(B_1)}\|w\|_{L^{\frac{np}{np-n+s}}(B_1)}\\
    &\leq c\mathcal{E}(u;B_1)^{\frac1p}\mathcal{E}(u;B_{1})^{\frac1{p'}}\widetilde{\mathcal{E}}(w;B_1)^{\frac1p}\\
    &\leq c\epsilon(\mathcal{E}(u;B_1)+\widetilde{\mathcal{E}}(w;B_1)),
\end{align*}
where we have employed H\"older's inequality and Lemma \ref{lem.spi}. On the other hand, we estimate $J_{2,4}$ as 
\begin{align*}
    |J_{2,4}|&\leq c\int_{\bbR^n\setminus B_{5/8}}\frac{|u(y)-(u)_{B_1}|^p}{|y|^{n+sp}}\,dy\int_{B_{1/2}}|w(x)|\,dx\\
    &\leq c\left[\int_{\bbR^n\setminus B_{1}}\frac{|u(y)-(u)_{B_1}|^q}{|y|^{n+sp}}\,dy+\int_{B_1}|u(y)-(u)_{B_1}|^p\,dy\right]\int_{B_{1/2}}|w(x)|\,dx,
\end{align*}
where we have used the fact that $|u|\equiv 1$ on $\bbR^n$.
We now use Lemma \ref{lem.spi}, \eqref{cond.usmall} and Young's inequality to obtain
\begin{align*}
    |J_{2,4}|\leq c(\mathcal{E}(u;B_1)^{\frac{q}p}+\mathcal{E}(u;B_1)){\mathcal{E}}(w;B_1)^{\frac1p}&\leq c\epsilon^{{q-(p-1)}}\mathcal{E}(u;B_1)^{\frac1{p'}}\widetilde{\mathcal{E}}(w;B_1)^{\frac1p}\\
    &\leq c\epsilon^{q-p+1}(\mathcal{E}(u;B_1)+\widetilde{\mathcal{E}}(w;B_1))
\end{align*}
for some constant $c=c(n,s,\Lambda,N,q)$.
Thus, combining all the estimate $J_{2,i}$ for each $i=1,2,3,4$, we get 
\begin{align}\label{estj2.decayu}
    |J_2|\leq c\epsilon^{q-p+1}(\mathcal{E}(u;B_1)+\widetilde{\mathcal{E}}(w;B_1))
\end{align}
for some constant $c=c(n,s,\Lambda,N)$.
By following the same lines as in the estimate of $J_2$ with $\overline{{u}}^j$ replaced by $u^j$, we have
\begin{align*}
    |J_3|\leq c\epsilon^{q-p+1}(\mathcal{E}(u;B_1)+\widetilde{\mathcal{E}}(w;B_1))
\end{align*}
for some constant $c=c(n,s,\Lambda,N)$. Using this together with \eqref{estj1.decayu} and \eqref{estj2.decayu}, we obtain
\begin{align*}
    |J|&\leq c\left[\epsilon^{q-p+1}\mathcal{E}(u;B_1)+(\epsilon^{q-p+1}+\delta)\widetilde{\mathcal{E}}(w;B_{1})+\delta^{-\frac1{p-1}}\left(\int_{B_{1}}|\mu|^{p_0}\right)^{\frac{p'}{p_0}}\right],
\end{align*}
where $c=c(n,s,\Lambda,N)$.
Therefore, we derive 
\begin{equation}\label{estl1.decayu}
\begin{aligned}
   |L_1|\leq c|J|
  \leq c\left[\epsilon^{q-p+1}\mathcal{E}(u;B_1)+(\epsilon^{q-p+1}+\delta)\widetilde{\mathcal{E}}(w;B_{1})+\delta^{-\frac1{p-1}}\left(\int_{B_{1}}|\mu|^{p_0}\right)^{\frac{p'}{p_0}}\right]
\end{aligned}
\end{equation}
where $c=c(n,s,\Lambda,q,p_0)$.

\textbf{Estimate of $L_2$}. First, we note that
\begin{align*}
    L_2=\sum_i\frac14\int_{\bbR^n}\int_{\bbR^n}|d_su|^p(u^i(x)-u^i(y))(w^i(x)-w^i(y))a(x,y)k(x,y)\,dx\,dy.
\end{align*}
Thus, using the fact that $w\equiv 0$ on $\bbR^n\setminus B_{1/2}$, We now estimate $L_2$ as 
\begin{align*}
    |L_2|&\leq c\int_{B_{5/8}}\int_{B_{5/8}}|d_su|^p|u(x)-u(y)||w(x)-w(y)|k(x,y)\,dy\,dx\\
    &\quad+c\int_{B_{1/2}}\int_{\bbR^n\setminus B_{5/8}}|d_su|^p|u(x)-u(y)||w(x)|k(x,y)\,dy\,dx\eqqcolon L_{2,1}+L_{2,2}
\end{align*}
for some constant $c=c(n,s,\Lambda,N)$.
By H\"older's inequality and \eqref{ineq02.decayu}, we obtain 
\begin{equation}\label{ineq22.decayu}
\begin{aligned}
    L_{2,1}\leq c[u]^{p+1}_{W^{\frac{s(p-1)}{p+1},\frac{p(p+1)}{p-1}}(B_{5/8})}[w]_{W^{s,p}(B_{5/8})}
    &\leq c[u]^{p+1}_{W^{\frac{s(p-1)}{p+1},\frac{p(p+1)}{p-1}}(B_{5/8})}\widetilde{\mathcal{E}}(w;B_1)^{\frac1p}.
\end{aligned}
\end{equation}
We now use Lemma \ref{lem.gagnir}  to see that
\begin{align}\label{ineq3.decayu}
    [u]^{p+1}_{W^{\frac{s(p-1)}{p+1},\frac{p(p+1)}{p-1}}(B_{5/8})}\leq c[u]^{2}_{BMO(B_1)}[u]^{p-1}_{W^{s,p}(B_{1})}\leq c\epsilon^2\mathcal{E}(u;B_1)^{\frac1{p'}},
\end{align}
where $c=c(n,s,N)$. Next, we plug \eqref{ineq3.decayu} into \eqref{ineq22.decayu} and then use Young's inequality to obtain
\begin{align*}
    L_{2,1}\leq c\epsilon (\mathcal{E}(u;B_1)+\widetilde{\mathcal{E}}(w;B_1)).
\end{align*}
As in the estimate of $J_2$ with $\widetilde{u}^j$ replaced by $u$, we get
\begin{align}\label{l22.decayu}
    L_{2,2}\leq c\left[\epsilon^{q-p+1}\mathcal{E}(u;B_1)+(\epsilon^{q-p+1}+\delta)\widetilde{\mathcal{E}}(w;B_{1})+\delta^{-\frac1{p-1}}\left(\int_{B_{1}}|\mu|^{p_0}\right)^{\frac{p'}{p_0}}\right]
\end{align}
Thus, by combining the estimates $L_{2,1}$ and $L_{2,2}$ and applying Cauchy's inequality, we obtain
\begin{align}\label{estl2.decayu}
    |L_2|\leq c\left[\epsilon^{q-p+1}\mathcal{E}(u;B_1)+(\epsilon^{q-p+1}+\delta)\widetilde{\mathcal{E}}(w;B_{1})+\delta^{-\frac1{p-1}}\left(\int_{B_{1}}|\mu|^{p_0}\right)^{\frac{p'}{p_0}}\right]
\end{align}
for some constant $c=c(n,s,\Lambda,N,q)$.

\textbf{Estimate of $L_3$}. By H\"older's inequality, Lemma \ref{lem.spi} and Young's inequality, we estimate $L_3$ as 
\begin{align}\label{estl3.decayu}
    |L_3|\leq c\delta^{-\frac1{p-1}}\|\mu\|^{p'}_{L^{p_0}(B_{1/2})}+c\delta\widetilde{\mathcal{E}}(w;B_{1}).
\end{align}
In light of the estimate of $L_3$, \eqref{estl1.decayu} and \eqref{estl2.decayu}, we get 
\begin{align}\label{ineqlast.decayu}
    \left|(\mathcal{L}u-\mathcal{L}v)w\right|\leq c\left[\epsilon^{q-p+1}\mathcal{E}(u;B_1)+(\epsilon^{q-p+1}+\delta)\widetilde{\mathcal{E}}(w;B_{1})+\delta^{-\frac1{p-1}}\left(\int_{B_{1}}|\mu|^{p_0}\right)^{\frac{p'}{p_0}}\right].
\end{align}
We are now ready to estimate $L$ given in \eqref{ineq21.decayu}. To this end, we split the proof into two cases.
\begin{itemize}
    \item $p\geq2$. From \eqref{estl.decayu}, \eqref{ineqlast.decayu} and \eqref{ineq02.decayu}, we derive 
\begin{align*}
    L\leq c\left[\epsilon^{q-p+1}\mathcal{E}(u;B_1)+(\epsilon^{q-p+1}+\delta)\widetilde{\mathcal{E}}(w;B_{1/2})+\delta^{-\frac1{p-1}}\left(\int_{B_{1}}|\mu|^{p_0}\right)^{\frac{p'}{p_0}}\right]
\end{align*}
for some constant $c=c(n,s,\Lambda,N,q,p_0)$.
Thus taking $\delta$ and $\epsilon_0$ sufficiently small depending only on $n,s,\Lambda,N,q$ and $p_0 $, we get
\begin{align}\label{ineqlast2.decayu}
     L\leq c\left[\epsilon^{q-p+1}\mathcal{E}(u;B_1)+\left(\int_{B_{1}}|\mu|^{p_0}\right)^{\frac{p'}{p_0}}\right],
\end{align}
where we have also used \eqref{ineq02.decayu}.
\item $p\leq 2$. After a few simple computations together with Young's inequality and \eqref{ineq02.decayu}, we estimate $L$ given in \eqref{estll.decayu} as
\begin{align*}
    L&\leq c\left(\mathcal{E}(u;B_1)+\widetilde{\mathcal{E}}(w;B_1))\right)^{\frac{2-p}{2}}\\
   &\quad\times \left[\epsilon^{q-p+1}\mathcal{E}(u;B_1)+(\epsilon^{q-p+1}+\delta)\widetilde{\mathcal{E}}(w;B_{1})+\delta^{-\frac1{p-1}}\left(\int_{B_{1}}|\mu|^{p_0}\right)^{\frac{p'}{p_0}}\right]^{\frac{p}{2}}\\
   &\leq c\left[\frac{\epsilon^{q-p+1}}{\delta^{\frac{2-p}{p}}}\mathcal{E}(u;B_1)+(\epsilon^{q-p+1}+\delta)\widetilde{\mathcal{E}}(w;B_{1/2})+\delta^{-\frac1{p-1}-\frac{2-p}{p}}\left(\int_{B_{1}}|\mu|^{p_0}\right)^{\frac{p'}{p_0}}\right].
\end{align*}
Thus taking $\delta$ and $\epsilon_0$ sufficiently small depending only on $n,s,\Lambda,N,q$ and $p_0 $, we get \eqref{ineqlast2.decayu} with $p\leq 2$.
\end{itemize}
Plugging \eqref{ineqlast2.decayu} into \eqref{ineq21.decayu} yields the desired estimate. This completes the proof. 
\end{proof}

 Now we recall the fractional maximal function of order $\beta$ for $|f|$ at $z$ given by 
\begin{equation} \label{maxfun}
    M_{\beta,R}(|f|)(z)\coloneqq \sup_{0<r\leq R}r^\beta\dashint_{B_r(z)}|f|\,dx.
\end{equation}
See \cite[Definition 2.2]{DieNow23} for more details. Using this and Lemma \ref{lem.decayu}, we now obtain the following oscillation type estimates.
\begin{lemma}\label{lem.decayu2}
Let $p=\frac{n}{s}$ and $p_0\in\left(1,1+\frac{s}{2(n-s)}\right]$. Fix $\vartheta< \min\left\{\alpha,\frac{sp}{p-1}\right\}$ with $\vartheta q<sp$, where the constants $q$ and $\beta$ are determined in \eqref{cond.q} and \eqref{lem.fracdecay}, respectively.
    Let $u\in \widehat{W}^{s,p}(B_{2R}(x_0);\bbR^N)$ with $|u|\equiv 1$ on $\bbR^n$ be a weak solution to 
    \begin{equation}\label{goal.decayu2}
        \mathcal{L}u=|\mathbf{d}_{s,a}u|^pu+\mu\quad\text{in }B_{R}(x_0).
    \end{equation}
    Then there is a sufficiently small $\epsilon=\epsilon(n,s,\Lambda,q,\vartheta,p_0)$ such that if 
    \begin{equation}\label{decayu2.small}
        [u]_{M^{p,0}(B_{2R}(x_0))}\leq \epsilon,
    \end{equation}
    then
    \begin{align*}
        \sup_{0<r\leq R}r^{-\vartheta}\dashint_{B_r(x_0)}|u-(u)_{B_{r}(x_0)}|\,dx\leq c\left(\epsilon/R^{\vartheta}+M_{p_0(sp-\vartheta(p-1)),R}(|\mu|^{p_0})(x_0)^{\frac{1}{p_0(p-1)}}\right)
    \end{align*}
    for some constant $c=c(n,s,\Lambda,N,q,\vartheta,p_0)$.
\end{lemma}
\begin{proof}
We now choose $\beta=\frac12\left(\min\left\{\alpha,\frac{sp}{p-1}\right\}+\vartheta\right)>\vartheta$ to see that $\beta\leq\alpha$ with $\beta q<sp$. Then there is a small integer $\rho_0=\rho_0(n,s,\Lambda,N,q,\vartheta,p_0)$ such that 
    \begin{align}\label{ineq.rho0}
        c\rho_0^{p(\beta-\vartheta)}\leq 1/128,
    \end{align}
    where the constant $c=c(n,s,\Lambda,N,q,p_0)$ is determined in Lemma \ref{lem.decayu}. We next choose $\epsilon=\epsilon(n,s,\Lambda,N,q,\vartheta,p_0)$ so that 
    \begin{align}\label{ineq.epsilon}
        c\mathcal{A}(\rho_0,\epsilon)\leq 1/64,
    \end{align}
    where $\mathcal{A}(\rho,\epsilon)$ is determined in Lemma \ref{lem.decayu}.
    Therefore, if 
    \begin{equation}\label{decayu2.small33}
        [u]_{M^{p,0}(B_{2R}(x_0))}\leq \epsilon,
    \end{equation}
    then by Lemma \ref{lem.decayu}, we obtain
\begin{align}\label{ineq1.decayu2}
        \frac{\mathcal{E}(u;B_{\rho_0 r}(x_0)))}{(\rho_0r)^{p\vartheta}}\leq \frac14 \frac{\mathcal{E}(u;B_{ r}(x_0)))}{r^{\vartheta}}+cM_{p_0(sp-\vartheta(p-1)),r}(|\mu|^{p_0})(x_0)^{\frac{p'}{p_0}}
    \end{align}
    for any $r\leq R$. Let us fix $r_0<R$ and define
    \begin{align*}
        N_{r_0}(x_0)\coloneqq \sup_{r_0\leq r\leq R}\frac{\mathcal{E}(u;B_{r}(x_0)))}{r^{p\vartheta}}.
    \end{align*}
    Using \eqref{ineq1.decayu2}, we deduce
    \begin{align*}
        N_{r_0}(x_0)\leq \frac{1}{4}N_{r_0}(x_0)+c\left(\mathcal{E}(u;B_R(x_0))/R^{p\vartheta}+M_{p_0(sp-\vartheta(p-1)),R}(|\mu|^{p_0})(z)^{\frac{p'}{p_0}}\right)
    \end{align*}
    for some constant $c=c(n,s,\Lambda,N,q,p_0)$. By taking $r_0\to 0$ and using the condition \eqref{decayu2.small33}, we have 
    \begin{align*}
        \sup_{0< r\leq R}\frac{\mathcal{E}(u;B_{r}(x_0)))}{r^{p\vartheta}}\leq c\left(\epsilon^p/R^{p\vartheta}+M_{p_0(sp-\vartheta(p-1)),R}(|\mu|^{p_0})(z)^{\frac{p'}{p_0}}\right)
    \end{align*}
    for some constant $c=c(n,s,\Lambda,N,q,p_0)$. Using this and Lemma \ref{lem.spi}, we derive the desired result.
\end{proof}

We are now ready to prove Theorem \ref{pot_intro} and Theorem \ref{thm0}.
\begin{proof}[Proof of Theorem \ref{pot_intro} and Theorem \ref{thm0}.] Let us fix $\delta=\min\left\{\alpha,\frac{sp}{p-1}\right\}$, where the constant $\alpha$ is determined in Lemma \ref{lem.fracdecay}. Given $\gamma<\frac{n}{sp-(p-1)\delta}$, we write $\vartheta\coloneqq \frac{1}{p-1}\left(sp-\frac{n}{\gamma}\right)<\delta$. We now choose 
\begin{equation}\label{q0.defn}
    q_0=\min\left\{\frac{(p-1)+sp/\vartheta}{2},\frac{2p-1}{2},\frac{p+1}{2}\right\}
\end{equation}
to see that $\vartheta  q_0< sp$.
Next, we select the constant $p_0=\min\left\{\frac{\gamma+1}{2},1+\frac{s}{2(n-s)}\right\}$ to see that $\mu\in L^{p_0}(B_{2R}(x_0);\bbR^N)$. 

We now want to show that there is a sufficiently small $R_0\leq R$ such that \eqref{decayu2.small} holds with $R=R_0$. To this end, note that there is sufficiently small $\rho=\rho(\delta,N)\leq R/16$ such that
    \begin{align*}
       \int_{B_{4\rho}(x_0)}\int_{B_{4\rho}(x_0)}|d_su|^p\frac{\,dx\,dy}{|x-y|^n}\leq \delta.
    \end{align*}
    We next prove that for any positive integer $l\geq1$ and $B_{r}(z)\subset B_{2^{-l}\rho}(x_0)$,
    \begin{align*}
        \mathrm{Tail}_{q_0}(u-(u)_{B_r}(z);B_{r}(z))\leq c(\delta+2^{-spl}),
    \end{align*}
    where $q_0=\min\{(p+1)/2,(2p-1)/2\}$ and $c=c(N,s)$.
   To do this, first we observe $B_{2^lr}(z)\subset B_{2\rho}(x_0)$ and
   \begin{align*}
        \int_{B_{2^{i}r}(z)}\int_{B_{2^{i}r}(z)}|d_su|^p\frac{\,dx\,dy}{|x-y|^n}\leq \int_{B_{2\rho}(x_0)}\int_{B_{2\rho}(x_0)}|d_su|^p\frac{\,dx\,dy}{|x-y|^n}\leq \delta
   \end{align*}
   for any $i\geq0$. Thus, using H\"older's inequality and Lemma \ref{lem.spi}, we derive 
   \begin{align*}
       \mathrm{Tail}_{q_0}(u-(u)_{B_{r}(z)};B_{r}(z))^q
       &\leq c\sum_{i=0}^{l}2^{-spi}\dashint_{B_{2^{i}r}(z)}|u-(u)_{B_{2^{i}r}(z)}|^{q_0}\,dx\\
        &\quad+c2^{-spl}\mathrm{Tail}_{q_0}(u-(u)_{B_{2^l}(z)};B_{2^lr}(z))^{q_0}\\
         &\leq c\sum_{i=0}^{k}2^{-2si}\left((2^ir)^{sp}\dashint_{B_{2^{i}r}}\int_{B_{2^{i}r}}|d_su|^p\frac{\,dx\,dy}{|x-y|^n}\right)^{\frac{p}{q_0}}\\
        &\quad+c2^{-spl}\left(\rho^{sp}\dashint_{B_{4\rho}(x_0)}\int_{B_{4\rho}(x_0)}|d_su|^p\frac{\,dx\,dy}{|x-y|^n}\right)^{\frac{p}{{q_0}}}\\
        &\quad+c2^{-spl}\left(\mathrm{Tail}_{q_0}(u-(u)_{B_{4\rho}(x_0)};B_{4\rho}(x_0))\right)\\
        &\leq c(\delta+2^{-spl}),
   \end{align*}
   where we have used $\|u\|\equiv 1$ and 
   \begin{align*}
       \mathrm{Tail}_{q_0}(u-(u)_{B_{4\rho}(x_0)};B_{4\rho}(x_0))\leq c
   \end{align*}
   for some constant $c=c(s,N)$.
   Therefore, by taking $\rho$ sufficiently small and taking $l$ sufficiently large, we have 
    \begin{align}\label{ineq.nspii}
[u]_{M^{p,0}(B_{R_0}(x_0))}&\leq \sup_{B_r(z)\subset B_{R_0}(x_0)}[u]^p_{W^{s,p}(B_{r}(z))}+\mathrm{Tail}_{q_0}(u-(u)_{B_{r}(z)};B_r(z))^p
\leq c(\delta+2^{-spl})^{\frac1p}\leq \epsilon,
    \end{align}
where $R_0=2^{-l}\rho$.

By Lemma \ref{lem.decayu2} with $q$ replaced by $q_0$ given in \eqref{q0.defn}, we have 
\begin{align*}
    \sup_{0<r\leq R_0}r^{-\vartheta}\dashint_{B_r(x_0)}|u-(u)_{B_{r}(x_0)}|\,dx\leq c\left(\epsilon/R_0^{\vartheta}+M_{p_0(sp-\vartheta(p-1)),R_0}(|\mu|^{p_0})(x_0)^{\frac{1}{p_0(p-1)}}\right).
\end{align*}
In particular, this proves Theorem \ref{pot_intro}.

By a simple modification of the proof of \cite[Proposition 2.5]{DieNow23}, we have 
\begin{align*}
    \sup_{0<r\leq R_0}r^{-\vartheta}\dashint_{B_r(x_0)}|u-(u)_{B_{r}(x_0)}|\,dx\leq c\left(\epsilon/R_0^{\vartheta}+\|\mu\|_{L^{\gamma,\infty}(B_{2R_0}(x_0))}\right).
\end{align*}
We now apply classical Campanato's theory as in \cite[Theorem 2.9]{Giu03} and a standard covering argument to arrive at $u\in C^{\vartheta}(B_R(x_0))$. This completes the proof of Theorem \ref{thm0}.
\end{proof}

%%%%%%%%%%%%%%%%%%%%%%%%%nonlinear/linear case

\section{Regularity for $p=2$}
In this section, we prove several improved regularity results for linear nonlocal harmonic map systems under suitable smallness assumptions.

First, we observe the following higher-order decay estimates for linear nonlocal systems. 
\begin{lemma}\label{lem.potsgra}
    Let $p=2$ and let $v\in \widehat{H}^{s}(B_1;\bbR^N)$ be a weak solution to 
    \begin{equation*}
        \mathcal{L}v=0\quad\text{in }B_1
    \end{equation*}
    with $a(x,y)=a(x-y)$. Let us fix $\alpha\in(s,\min\{2s,1\})$ such that $\alpha q<2s$, where $q$ is determined in \eqref{cond.q} with $p=2$.
    Then, for any $\tau\in(0,1/4]$,
    \begin{align}\label{res2.potsgra}
       \mathcal{E}(v;B_{\tau})\leq c\tau^{2\alpha}\mathcal{E}(v;B_{1/2})
    \end{align}
    where the constant $c$ depends only on $n,s,\Lambda,N,q$ and $\alpha$.
\end{lemma}
\begin{proof}
Since $v^i-(v^i)_{B_{3/8}}\in \widehat{H}^s(B_1)$ is a weak solution to
\begin{align*}
    \mathcal{L}(v^i-(v^i)_{B_{3/8}})=0\quad\text{in }B_1,
\end{align*}
by {\cite[Theorem 1.7]{Now21}}, we have
\begin{align}\label{ineq2.potsgra}
    \|v^i\|_{C^{\alpha}(B_{1/4})}\leq  c\left(\dashint_{B_{3/8}}|v^i-(v^i)_{B_{3/8}}|\,dx+\mathrm{Tail}_1(v^i-(v^i)_{B_{3/8}};B_{3/8})\right)
\end{align}
for some constant $c=c(n,s,\Lambda,\alpha,q)$.
We now follow the same lines as in the proof of Lemma \ref{lem.fracdecay} together with the fact that $\alpha q<2s$, in order to deduce the desired estimate.
\end{proof}
Using Lemma \ref{lem.potsgra} and Remark \ref{rmk.h1norm}, we obtain a corresponding result to Lemma \ref{lem.decayu2}.
\begin{lemma}\label{lem.decayu3}
Let $p=2$, $n>2s$ and $p_0=\min\left\{\frac{n}{2s},\frac{2n}{n+s}\right\}$. Fix $\alpha\in(s,\min\{2s,1\})$ with $\alpha q<2s$, where the constant $q$ is determined in \eqref{cond.q}.
    Let $u\in \widehat{W}^{s,2}(B_{2R}(x_0);\bbR^N)$ with $|u|\equiv 1$ on $\bbR^n$ be a weak solution to 
    \begin{equation}\label{goal.decayu2}
        \mathcal{L}u=|\mathbf{d}_{s,a}u|^2u+\mu\quad\text{in }B_{R}(x_0).
    \end{equation}
    There is a sufficiently small $\epsilon=\epsilon(n,s,\Lambda,q,\alpha,p_0)$ such that if 
    \begin{equation*}
        [u]_{M^{2,0}(B_{2R}(x_0))}\leq \epsilon,
    \end{equation*}
    then
    \begin{align}\label{goal.decayu3}
        \sup_{0<r\leq R}r^{-\alpha}\dashint_{B_r(x_0)}|u-(u)_{B_{r}(x_0)}|\,dx\leq c\left(\epsilon/R^{\alpha}+M_{p_0(2s-\alpha),R}(|\mu|^{p_0})(x_0)^{\frac{1}{p_0(p-1)}}\right)
    \end{align}
    for some constant $c=c(n,s,\Lambda,N,\alpha,p_0)$.
\end{lemma}
\begin{proof}
    There is a small constant $\alpha_0>\alpha$ such that $\alpha_0\in(s,\min\{2s,1\})$ with $\alpha_0 q<2s$. First, we note that Lemma \ref{lem.h1norm2} with $p=2$ and $p_0=\min\left\{\frac{n}{2s},\frac{2n}{n+s}\right\}$ holds by proceeding along the same lines as in the proof of Lemma \ref{lem.h1norm2} together with Remark \ref{rmk.h1norm}. Therefore, now using this and Lemma \ref{lem.potsgra}, we obtain the same results given in Lemma \ref{lem.decayu} with $p=2$ and $\beta=\alpha_1$ without changing the proof. Therefore, applying the same approach as in the proof of Lemma \ref{lem.decayu2} with $p=2$ yields \eqref{goal.decayu3}.
\end{proof}

Before we give the proof of Theorem \ref{thm1}, we observe the following lemma.
\begin{lemma}\label{lem.dsuan}
    Let $\beta\in(s,1)$, $|u|\equiv 1$ on $\bbR^n$ and $a(x,y)=a(x-y)$. If $u\in C^{\beta}(B_{2R}(x_0))$, then we have
    \begin{equation}\label{goal1.dsuan}
        g(x)\coloneqq \left(|\mathbf{d}_{s,a}u|^2u\right)(x)\in L^\infty(B_{R/2}(x_0)).
    \end{equation}
\end{lemma}
\begin{proof}
First, we rewrite the function $g(x)$ as
\begin{align*}
    g(x)=\left[\int_{\bbR^n}\frac{|u(x+y)-u(x)|^2}{|y|^{n+sp}}a(y)\,dy\right]u(x).
\end{align*}
Since $u\in C^{\beta}(B_{2R}(x_0))$, we have 
\begin{align*}
    |g(x)|\leq c\left(\int_{B_{R}}[u]^2_{C^{\beta}(B_{2R}(x_0))}\frac{\,dy}{|y|^{n+2s-2\beta}}+\int_{\bbR^n\setminus B_R}\frac{\,dy}{|y|^{n+sp}}\right)<\infty 
\end{align*}
for any $x\in B_{R/2}(x_0)$, which gives \eqref{goal1.dsuan}.
\end{proof}

Using Lemma \ref{lem.decayu3} and Lemma \ref{lem.dsuan}, we are now ready to prove the optimal H\"older regularity of the solution in terms of the right-hand side under our smallness assumption on the solution.
\begin{proof}[Proof of Theorem \ref{thm1}.]
Let us fix $p_0=\min\left\{\frac{n}{2s},\frac{2n}{n+s}\right\}$. Let $\mu \in L^{\gamma}(B_{2R}(x_0);\bbR^N)$. We now prove our main result according to the range $\gamma$.
\begin{enumerate}
    \item Let $\gamma\in(n/2s,n/s]$. We now choose 
    \begin{equation}\label{choi1.thm1}
        q\coloneqq 5/4,\quad \alpha\coloneqq 7s/6\quad\text{and}\quad \vartheta\coloneqq 2s-n/\gamma.
    \end{equation}
    Next, we fix $R_0=R/4$. Then we observe $u$ is a weak solution to \eqref{goal.decayu2} with $R=R_0$. First, by considering the choice of the parameter $\rho_0$ determined in \eqref{ineq.rho0} and considering \eqref{choi1.thm1}, we deduce 
    \begin{equation*}
        c\rho_0^{2\alpha-2\vartheta}\leq c\rho_0^{s/6}\leq 1/128,
    \end{equation*}
    where $\rho_0=\rho_0(n,s,\Lambda,N,\gamma)$, as the constant $c$ depends only on $n,s,\Lambda$ and $N$. Therefore, we can also choose $\epsilon_0=\epsilon_0(n,s,\Lambda,N,\gamma)$ determined in \eqref{ineq.epsilon} so that 
    \begin{align*}
        \sup_{0<r\leq R_0}r^{-\vartheta}\dashint_{B_r(x_0)}|u-(u)_{B_{r}(z)}|\,dx\leq c\left(\epsilon_0/R_0^{\vartheta}+M_{p_0(4s-2\vartheta),R_0}(|\mu|^{p_0})(x_0)^{\frac{1}{p_0}}\right)
    \end{align*}
    for some constant $c=c(n,s,\Lambda,N,\gamma)$. By \cite[Proposition 2.5]{DieNow23}, we have 
    \begin{align*}
        \sup_{0<r\leq R_0}r^{-\vartheta}\dashint_{B_r(x_0)}|u-(u)_{B_{r}(z)}|\,dx\leq c\left(\epsilon_0/R_0^{\vartheta}+\|\mu\|_{L^{\gamma}(B_{2R_0}(x_0))}\right).
    \end{align*}
    As in the last lines of the proof of Theorem \ref{thm0}, we deduce $u\in C^{\vartheta}(B_R(x_0);\bbR^N)$ whenever
    \begin{equation*}
        [u]_{M^{2,0}(B_{2R}(x_0))}\leq \epsilon_0
    \end{equation*}
    for some $\epsilon_0=\epsilon_0(n,s,\Lambda,N,\gamma)$.
    \item $\gamma>n/s$.
    Let us choose 
    \begin{equation}\label{choi2.thm1}
        q\coloneqq 5/4,\quad \alpha\coloneqq 7s/6\quad\text{and}\quad \vartheta\coloneqq \min\{13s/12,2s-n/p\}.
    \end{equation}
    As in the proof of the case when $\gamma\in(n/2s,n/s]$, we have if 
    \begin{equation*}
        [u]_{M^{2,0}(B_{2R}(x_0))}\leq \epsilon_1
    \end{equation*}for some $\epsilon_1=\epsilon_1(n,s,\Lambda,N,p)$, then we have $u\in C^{\vartheta}(B_R(x_0))$, where $\vartheta>s$. By Lemma \ref{lem.dsuan}, we have that $u$ is a weak solution to 
    \begin{align*}
        \mathcal{L}u=g+\mu\quad\text{in } B_{R/4(x_0)},
    \end{align*}
    where $g\in L^\infty(B_{R/4}(x_0);\bbR^N)$ and $\mu\in L^{\gamma}(B_{R/4}(x_0);\bbR^N)$. Since $u^i$ is a weak solution to
    \begin{align*}
        \mathcal{L}u^i=g^i+\mu^i\quad\text{in }B_{R/4}(x_0),
    \end{align*}
    by using \cite[Theorem 2.4.3]{FerRos24}, we get $u\in C^{2s-n/\gamma}(B_{R/8}(x_0))$ when $\gamma\neq\frac{n}{2s-1}$. In addition, when $\mu\in L^{\frac{n}{2s-1},1}(B_{2R}(x_0))$ with $s>1/2$, then by \cite[Corollary 1.11]{KuuSimYan22}, we have $u^i\in C^{1}(B_{R/8}(x_0))$.
\end{enumerate}
Lastly, we note from Lemma \ref{lem.reltildeee} and \eqref{defn.m20} that 
\begin{align}\label{ineqaa}
    [u]^2_{M^{2,0}(B_{2R}(x_0))}\leq c\sup_{B_\rho(z)\subset B_{2R}(x_0)}\rho^{2s-n}\widetilde{\mathcal{E}}(u;B_\rho(z))
\end{align}
for some constant $c=c(n,s,\Lambda,N)$, where the constant $q=5/4$ is determined in \eqref{choi1.thm1} and \eqref{choi2.thm1}. Therefore, by taking $\epsilon$ sufficiently small, we get the desired partial regularity result. This completes the proof.
    
\end{proof}
We now use Lemma \ref{lem.decayu3} together to derive a higher Sobolev regularity of the solution.
\begin{theorem}\label{thm2}
    Let $n>2s$ and let $u$ be a weak solution to 
    \begin{equation*}
        \mathcal{L}u=|\mathbf{d}_{s,a}u|^2u+\mu\quad\text{in }B_{2R}(x_0)
    \end{equation*} with $a(x,y)=a(x-y)$. Let us fix $\sigma\in(0,\min\{1,2s\})$ and $\gamma\geq2$. If 
    \begin{align*}
        [u]_{M^{2,0}(B_{2R}(x_0))}\leq \epsilon
    \end{align*}
    for some small $\epsilon=\epsilon(n,s,N,\Lambda,\sigma,q)$ and $\mu\in L^{\gamma}(B_{2R}(x_0);\bbR^N)$ with $q\geq2$, then we have $u\in W^{\sigma,\frac{n\gamma}{n-(2s-\sigma)\gamma}}(B_R(x_0);\bbR^N)$.
\end{theorem}

\begin{proof}
   We fix $q\geq2$, $\sigma\in(0,\min\{1,2s\})$, $\sigma'=(\min\{1,2s\}+\sigma)/2>\sigma$ and $p_0=\min\left\{\frac{n}{2s},\frac{2n}{n+s}\right\}$. By taking $q<3/2$ sufficiently small so that $\sigma' q<2s$. Then by Lemma \ref{lem.decayu2}, there is a sufficiently small $\epsilon_0=\epsilon_0(n,s,\Lambda,N,\sigma)$ such that
    if 
    \begin{equation}\label{ineqthh14}
        [u]_{M^{2,0}(B_{2R}(x_0))}\leq \epsilon_0,
    \end{equation}
    then
    \begin{align*}
        \sup_{0<r\leq R}r^{-\sigma_0}\dashint_{B_r(x_0)}|u-(u)_{B_{r}(x_0)}|\,dx\leq c\left(\epsilon/R^{\sigma_0}+\left(M_{p_0(2s-\sigma_0),R}(|\mu|^{p_0})(x_0)\right)^{\frac1{p_0}}\right),
    \end{align*}
    where $c=c(n,s,\Lambda,N,\sigma)$ and $\sigma_0=(\sigma+\sigma')/2$. Since $\mu\in L^\gamma$,  by \cite[Proposition 2.8]{DieNow23}, we deduce  
    \begin{equation*}
        \left(M_{p_0(2s-\sigma),R}(|\mu|^{p_0})(x)\right)^{\frac1{p_0}}\in L^{\frac{n\gamma}{n-\gamma(2s-\sigma_0)}}.
    \end{equation*}Therefore, by \cite[Proposition 2.8]{DieNow23}, we have $u\in W^{\sigma,\frac{n\gamma}{n-(2s-\sigma)\gamma}}(B_{R}(x_0);\bbR^N)$ with the estimate
    \begin{align}\label{est1.thm2}
        [u]_{W^{\sigma,\frac{n\gamma}{n-(2s-\sigma)\gamma}}(B_{R}(x_0))}\leq c\left(1+\|\mu\|_{L^{\gamma}(B_{2R}(x_0))}\right)
    \end{align}
    for some constant $c=c(n,s,\Lambda,N,\sigma,\gamma,R)$. This completes the proof.
\end{proof}

Before, we give the proof of Theorem \ref{thm3}, we observe the following fractional Calder\'on-Zygmund type estimates. Note that our argument is based on the proof of \cite[Theorem 1.3]{BicWarZua17}.
\begin{lemma}\label{lem.2sr}
    Let us fix $\gamma\geq2$ and let $u\in W^{s,2}(B_{2R}(x_0))\cap W^{s,\gamma}(B_{2R}(x_0))$ with $|u|\equiv 1$ on $\bbR^n$ be a weak solution to
    \begin{align*}
        (-\Delta)^su=f\quad\text{in }B_R(x_0),
    \end{align*}
    where $f\in L^\gamma(B_{2R}(x_0))$. Then we have 
    \begin{align*}
        [u]_{W^{2s,\gamma}(B_{R/4}(x_0))}\leq \left([u]_{W^{s,\gamma}(B_{R}(x_0))}+\|f\|_{L^\gamma(B_{R}(x_0))}+1\right)
    \end{align*}
    for some constant $c=c(n,s,\gamma,R)$. In addition, if $\gamma>1$ and $s=1/2$, we have the same result.
\end{lemma}
\begin{proof}
We may assume $R=1$ and $x_0=0$.
    Let us choose $\psi\in C_c^\infty(B_{1/2})$ with $\psi\equiv 1$ on $B_{1/4}$. Now we want to observe that $v=u\psi$ is a weak solution to 
    \begin{equation}\label{eq.2sr}
        (-\Delta)^sv=f\psi+g\quad\text{in }\bbR^n,
    \end{equation}
    where $g\in L^\gamma(\bbR^n)$. First, we note from the equation (3.2) given in the proof of \cite[Theorem 1.3]{BicWarZua17} that 
    \begin{align*}
        (-\Delta)^s v=\psi(-\Delta)^su+u(-\Delta)^s\psi-I_s(u,\psi),
    \end{align*}
    where 
    \begin{equation*}
        I_s(u,\psi)(x)\coloneqq \int_{\bbR^n}\frac{(u(x)-u(y))(\psi(x)-\psi(y))}{|x-y|^{n+2s}}\,dy.
    \end{equation*}
    Let us write 
    \begin{equation*}
        g(x)\coloneqq (u(-\Delta)^s\psi)(x)-I_s(u,\psi)(x)\eqqcolon g_1(x)+g_2(x)
    \end{equation*}
    to see that $g_1,g_2\in L^p(\bbR^n)$. First, using the fact that $\psi\equiv 0$ on $\bbR^n\setminus B_{1/2},$ we estimate 
    \begin{align*}
        \|g_1\|^p_{L^p(\bbR^n)}&\leq c\int_{B_1}\left|u(x)\int_{\bbR^n}\frac{\psi(x)-\psi(y)}{|x-y|^{n+2s}}\,dy\right|^\gamma\,dx\\
        &\quad+c\int_{\bbR^n\setminus B_1}\left|u(x)\int_{B_1}\frac{\psi(x)-\psi(y)}{|x-y|^{n+2s}}\,dy\right|^\gamma\,dx\eqqcolon \sum_{i=1}^2I_i,
    \end{align*}
    where $c=c(p)$. We note that 
    \begin{align*}
        \int_{\bbR^n}\frac{\psi(x)-\psi(y)}{|x-y|^{n+2s}}\,dy=\int_{\bbR^n}\frac{\psi(x)-\psi(x+z)}{|z|^{n+2s}}\,dy=\int_{\bbR^n}\frac{\psi(x)-\psi(x-z)}{|z|^{n+2s}}\,dy.
    \end{align*}
    Thus, we have 
    \begin{align*}
        \int_{\bbR^n}\frac{\psi(x)-\psi(y)}{|x-y|^{n+2s}}\,dy=\frac12\int_{\bbR^n}\frac{2\psi(x)-\psi(x+z)-\psi(x-z)}{|z|^{n+2s}}.
    \end{align*}
    Since $\psi\in C_c^\infty(B_{1/2})$, we have 
    \begin{align*}
        \left\|\int_{\bbR^n}\frac{\psi(x)-\psi(y)}{|x-y|^{n+2s}}\,dy\right\|_{L^\infty(\bbR^n)}\leq c
    \end{align*}
    for some constant $c=c(n,s)$. Using this, we have 
    \begin{align*}
        I_1\leq c\int_{B_1}|u(x)|^\gamma\,dx\leq c,
    \end{align*}
    where $c=c(n,s,p)$. We next estimate $I_2$ as 
    \begin{align*}
        I_2&\leq c\int_{\bbR^n\setminus B_1}\left(\int_{B_{1/2}}\frac{|u(x)|}{|x|^{n+2s}}|\psi(y)|\,dy\right)^\gamma\,dx\leq c,
    \end{align*}
    for some constant $c=c(n,s,p)$, where we have used the fact that $\psi\in C_c^\infty(B_{1/2})$ and 
    \begin{align*}
        |x|\leq 16|x-y|\quad\text{for any }x\in \bbR^n\setminus B_1\text{ and }y\in B_{1/2}
    \end{align*}
    Therefore, we have 
    \begin{equation}\label{g1.ineq}
        \|g_1\|_{L^\gamma(\bbR^n)}\leq c.
    \end{equation}
    We next observe from H\"older's inequality
    \begin{align*}
        \|g_2\|_{L^\gamma(\bbR^n)}&\leq c\int_{B_1}\left(\int_{B_1}\frac{|u(x)-u(y)|^\gamma}{|x-y|^{n+s\gamma}}\,dy\right)\left(\int_{B_1}\frac{|\psi(x)-\psi(y)|^{\gamma'}}{|x-y|^{n+s\gamma'}}\,dy\right)^{\frac{\gamma}{\gamma'}}\,dx\\
        &\quad+c\int_{\bbR^n\setminus B_1}\left|\int_{B_1}\frac{(u(x)-u(y))(\psi(x)-\psi(y))}{|x-y|^{n+2s}}\,dy\right|^\gamma\,dx\\
        &\quad+c\int_{B_1}\left|\int_{\bbR^n\setminus B_1}\frac{(u(x)-u(y))(\psi(x)-\psi(y))}{|x-y|^{n+2s}}\,dy\right|^\gamma\,dx\eqqcolon \sum_{i=1}^3J_i.
    \end{align*}
    Since $\psi\in C_C^\infty(B_{1/2})$, we have 
    \begin{align*}
        J_1\leq c[u]_{W^{s,\gamma}(B_1)}^p,
    \end{align*}
    and
    \begin{align*}
        J_2&\leq c\int_{\bbR^n\setminus B_1}\left|\int_{B_{1/2}}\frac{(u(x)-u(y))\psi(x)}{|x-y|^{n+2s}}\,dy\right|^\gamma\,dx+c\int_{\bbR^n\setminus B_1}\left|\int_{B_{1/2}}\frac{(u(x)-u(y))\psi(y)}{|x-y|^{n+2s}}\,dy\right|^\gamma\,dx.
    \end{align*}
    After a few simple computations together with the fact that $|u|\equiv 1$ on $\bbR^n$, we have 
    \begin{align*}
        J_2\leq c.
    \end{align*}
    Similarly, we have $J_3\leq c$. Thus, we have
    \begin{equation}\label{g2.ineq}
        \|g_2\|_{L^\gamma(\bbR^n)}\leq c([u]_{W^{s,\gamma}(B_1)}+1)
    \end{equation}
    for some constant $c=c(n,s,\gamma)$.
    Therefore, by \eqref{g1.ineq} and \eqref{g2.ineq} we have 
    \begin{align*}
        \|g\|_{L^\gamma(\bbR^n)}\leq c([u]_{W^{s,\gamma}(B_1)}+1)
    \end{align*}
    Since $(\psi(-\Delta)^su)(x)=(\psi f)(x)$, we get \eqref{eq.2sr}. Now using the standard regularity theory given in \cite{Ste70} or \cite[Theorem 2.7]{BicWarZua17}, we have 
    \begin{align*}
        [u\psi]_{W^{2s,\gamma}(\bbR^n)}\leq c\left(\|f\psi\|_{L^\gamma(\bbR^n)}+\|g\|_{L^\gamma(\bbR^n)}\right),
    \end{align*}
    which implies 
    \begin{align*}
        [u]_{W^{2s,\gamma}(B_{1/4})}\leq c\left(\|f\|_{L^\gamma(B_{1})}+[u]_{W^{s,\gamma}(B_1)}+1\right)
    \end{align*}
    for some constant $c=c(n,s,\gamma)$.
\end{proof}
By combining Theorem \ref{thm2} and Lemma \ref{lem.2sr}, we are able to prove Theorem \ref{thm3}.
\begin{proof}[Proof of Theorem \ref{thm3}]
    Let us fix $B_{2R}(x_0)\Subset \Omega$. Since $\gamma>n/2s$, there is a constant $\sigma=\sigma(n,s,\gamma)>s$ such that 
    \begin{align*}
        \frac{n\gamma}{n-(2s-\sigma)\gamma}>2\gamma.
    \end{align*}
    By Theorem \ref{thm2}, \eqref{est1.thm2} and \eqref{ineqaa}, there is a sufficiently small $\epsilon=\epsilon(n,s,\gamma,N)$ so that if 
    \begin{align}\label{conddd.thm17}
        \sup_{B_{\rho}(z)\Subset B_{2R}(x_0)}\rho^{2s-n}\widetilde{\mathcal{E}}(u;B_{\rho}(z))\leq \epsilon,
    \end{align}
    then 
    \begin{align*}
        [u]_{W^{\sigma,\frac{n\gamma}{n-(2s-\sigma)\gamma}}(B_R(x_0))}\leq c(1+\|\mu\|_{L^\gamma(B_{2R}(x_0))})
    \end{align*}
    for some constant $c=c(n,s,\gamma,N,R)$. Since $\sigma>s$ and $\frac{n\gamma}{n-(2s-\sigma)\gamma}>2\gamma$, we are able to prove 
    \begin{align*}
        |\mathbf{d}_{s}u|^2(x)\in L^\gamma(B_{R/2}(x_0)).
    \end{align*}
    To do this, we observe
    \begin{align*}
        \int_{B_{R/2}(x_0)}|\mathbf{d}_su|^{\gamma}(x)&\leq c\int_{B_{R/2}(x_0)}\left(\int_{B_{R}(x_0)}\frac{|u(x)-u(y)|^2}{|x-y|^{n+2s}}\,dy\right)^{\gamma}\,dx\\
        &\quad +c\int_{B_{R/2}(x_0)}\left(\int_{\bbR^n\setminus B_{R}(x_0)}\frac{|u(x)-u(y)|^2}{|x-y|^{n+2s}}\,dy\right)^{\gamma}\,dx\eqqcolon I_1+I_2.
    \end{align*}
    In addition, by the Sobolev embedding as in \cite[Proposition 2.5]{Now23v}, we have
    \begin{align*}
        [u]_{W^{s+\delta/2,2\gamma}(B_{R}(x_0))}\leq c[u]_{W^{\sigma,\frac{n\gamma}{n-(2s-\sigma)\gamma}}(B_R(x_0))},
    \end{align*}
    where $\delta\coloneqq (s-\sigma)/16$ and $c=c(n,s,\gamma,N,R)$.
    Then by H\"older's inequality, we have 
    \begin{align*}
        I_1&\leq \int_{B_{R/2}(x_0)}\left(\int_{B_R(x_0)}\frac{|u(x)-u(y)|^{2\gamma}}{|x-y|^{n+\gamma(2s+\delta)}}\,dy\right)\left(\int_{B_R(x_0)}\frac{1}{|x-y|^{n-\delta \gamma'}}\,dy\right)^{\frac{\gamma}{\gamma'}}\,dx\\
        &\leq c[u]_{W^{s+\delta/2,2\gamma}(B_{R}(x_0))}\\
        &\leq c(1+\|\mu\|_{L^\gamma(B_{2R}(x_0)})
    \end{align*}
    for some constant $c=c(n,s,\gamma,N,R)$. On the other hand, we estimate $I_2$ as 
    \begin{align*}
        I_2\leq c\int_{B_{R/2}(x_0)}\left(\int_{\bbR^n\setminus B_{R}(x_0)}\frac{1}{|y|^{n+s\gamma}}\,dy\right)^\gamma\,dx\leq c,
    \end{align*}
    where $c=c(n,s,\gamma,N)$. Combining the estimates $I_1$ and $I_2$ yields 
    \begin{align*}
        \int_{B_{R/2}(x_0)}(|\mathbf{d}_su|^2(x))^\gamma\leq c(1+\|\mu\|_{L^\gamma(B_{2R}(x_0)}),
    \end{align*}
    where $c=c(n,s,\gamma,N)$. Since $u^i$ is a weak solution to 
    \begin{align*}
        (-\Delta)^su^i=|\mathbf{d}_su|^2u^i+\mu^i\quad\text{in }B_{R/2}(x_0),
    \end{align*}
    where  $|\mathbf{d}_su|^2u^i,\mu^i\in L^\gamma(B_{R/2}(x_0))$, 
    by applying Lemma \ref{lem.2sr}, we have 
    \begin{align}\label{ineq22.thm3}
        [u^i]_{W^{2s,\gamma}(B_{R/8}(x_0))}\leq c(1+\|\mu^i\|_{L^\gamma(B_{R/2}(x_0)}),
    \end{align}
    where $c=c(n,s,\gamma,R)$. Since the inequality \eqref{ineq22.thm3} holds for every $i=1,2,\cdots, N$, we have  
    \begin{align}\label{ineq2.thm3}
        [u]_{W^{2s,\gamma}(B_{R/8}(x_0))}\leq c(1+\|\mu\|_{L^\gamma(B_{R/2}(x_0)}),
    \end{align} whenever \eqref{conddd.thm17} holds.
\end{proof}

Then using Theorem \ref{thm1} and Theorem \ref{thm3} together with \eqref{ineq.nspii} with $p=2$, we directly obtain the results given in Theorem \ref{cor2}.

As an application of Theorem \ref{thm2} and Theorem \ref{thm3}, we now prove higher differentiability of solutions to the fractional harmonic map heat flow.
\begin{proof}[Proof of Theorem \ref{thm4} and Corollary \ref{cor3}.]
    By the definition, we observe that for a.e. $t\in(0,T]$, $u$ is a weak solution to 
\begin{align*}
    (-\Delta)^su=|\mathbf{d}_{s}u|^2u-\partial_tu\quad\text{in }\Omega.
\end{align*}
Let us fix $\sigma\in(0,\min\{1,2s\})$ and $\sigma'\coloneqq (\sigma+\min\{1,2s\})/2$.
Let us choose $\epsilon=\epsilon(n,s,N,\sigma')$ which is determined in the proof of Theorem \ref{thm2} with $p=2$. 
Then, by the assumption \eqref{ass.thm4} given in Theorem \ref{thm2} and \eqref{est1.thm2}, we have 
\begin{align*}
        [u(\cdot,t)]^2_{W^{\sigma',\frac{2n}{n-2(2s-\sigma')}}(B_{R/4}(x_0))}\leq c(1+\|\partial_t u(\cdot,t)\|^2_{L^2(B_{2R}(x_0))})
    \end{align*}
    for some constant $c=c(n,s,N,\sigma,R)$, when $\epsilon=\epsilon(n,s,N,\sigma)$ is sufficiently small. 
    By the fractional Sobolev embedding as in \cite[Proposition 2.5]{Now23v}, we note that 
    \begin{align*}
         [u(\cdot,t)]_{W^{\sigma,2}(B_{R/4}(x_0))}\leq c[u(\cdot,t)]^2_{W^{\sigma',\frac{2n}{n-2(2s-\sigma')}}(B_{R/4}(x_0))}
    \end{align*}
    for some constant $c=c(n,s,N,\sigma,R)$.
    Therefore, we derive 
    \begin{align}\label{ineq1.thm3}
        \int_{I_{R/4}(t_0)}[u(\cdot,t)]^2_{W^{\sigma,2}(B_{R/4}(x_0))}\leq c\left(1+\int_{I_{R/4}(t_0)}\|\partial_tu(\cdot,t)\|^2_{L^2(B_{2R}(x_0))}\right).
    \end{align}
    In addition, if $n<4s$, then by \eqref{ineq2.thm3} we have 
    \begin{align*}
        [u(\cdot,t)]_{W^{2s,2}(B_{R/8}(x_0))}\leq c(1+\|\partial_t u(\cdot,t)\|_{L^2(B_{R/2}(x_0)}),
    \end{align*}
for a.e., $t\in I_{R/8}(t_0)$, provided that $\epsilon=\epsilon(n,s,N)$ is sufficiently small.  Using this, we derive the desired result, which completes the proof of Theorem \ref{thm4}. 
    
    Now we assume $n=1$ and $s=1/2$. Then as in the argument given in \eqref{ineq.nspii} together with the fact that $u\in L^\infty(0,T;\widehat{H}^s(\Omega;\bbR^N))$,  we have  for any $\delta>0$, there is a small constant $R_0=R_0(s,N)$ such that
    \begin{align*}
       \sup_{B_{\rho}(x)\subset B_{R_0}(x_0)}\,\sup_{t\in I_{R}(t_0)}\widetilde{\mathcal{E}}(u;B_{\rho}(x))\leq \delta.
    \end{align*}
    Thus, by taking $\delta=\delta(s,N)$ sufficiently small, we deduce 
    \begin{align*}
        \|u(\cdot,t)\|_{W^{1,2}(B_{R_0/4}(x_0))}\leq c\left(1+\|\mu(\cdot,t)\|_{L^2(B_{R_0/2}(x_0))}\right)
    \end{align*}
    for a.e., $t\in I_{R_0}(t_0)$, where $c=c(s,N)$. Therefore, we get
    \begin{align*}
        \|u\|_{L^2(I_{R_0/4}(t_0);W^{1,2}(B_{R_0/4}(x_0)))}\leq c\left(1+\|\mu\|_{L^2(Q_{R_0/2}(z_0))}\right),
    \end{align*}
    which completes the proof of Corollary \ref{cor3}.
    \end{proof}

   \section{Applications to harmonic maps with free boundaries and geometric flows of Plateau-type}\label{sec:applications}
    
    As explained in the introduction, fractional harmonic maps (at least for the case $s=1/2$) are related to a class of harmonic maps with free boundary. In this section, we explain in more details such connection, draw some relevant statements for geometric problems as consequences of some of our results and expand on some possible generalizations and open problems. 
    
    A major tool in the study of geometric problems is the existence of a monotone quantity. In the case of the standard Dirichlet energy, this amounts to its localized rescaled version being monotone in the radius and in the case of the parabolic heat flow a weighted version of the Dirichlet energy in some space-time cylinders. The existence of such quantity has important consequences for regularity, structure of the singular set, quantitative differentiation aspects, dimension reduction \`a la Federer just to name a few. When the problem has some conformal invariance ($n=2$ for the classical case), monotonicity is actually missing but the conformality is providing a revised tool to handle all these aspects. 
    
    A major drawback of our general energies is the lack of monotonicity in general dimensions as well as the range of parameters $(s,p)$. Nevertheless, in the case of the critical dimension and power, i.e. $n=sp$ our results indeed provide new estimates as stated in the introduction. Note that in the case of a general dimension $n\geq1$ and the constant coefficient version of $\mathcal L_s$, i.e. the fractional Laplacian there is indeed a monotonicity formula in terms of the Caffarelli-Silvestre extension, which indeed allows to prove partial regularity as developed by Millot-Pegon-Schikorra \cite{MilPegSch21}. For our general operators, Caffarelli-Silvestre is not available and at the moment there is no useful monotone quantity. 
    Another important aspects of fractional harmonic maps is the full convergence of blow-ups by means of the triviality of the defect measure in terms of a Marstrand argument as discovered in \cite{MR3900821} and exploited in \cite{MilPegSch21}. This aspect is also related to the existence of a monotone quantity (to apply Marstrand structure theorem of densities) and as a consequence we do not know if such compactness holds for our generalized maps. We conjecture however that such compactness results should hold without the structural assumptions of the operator being the fractional Laplacian.  
    
    We now turn to the application of our results to the case of harmonic maps with free boundary. Though half-harmonic maps have been introduced by Da Lio and Riviere in early 2000, the third author with Millot in \cite{millotSire} uncovered that they correspond to a new formulation of harmonic mappings with free boundary, a classical topic in geometric analysis. We introduce below the class of free boundary maps under consideration. 
    
     \begin{definition}\label{defFBHarm}
Let  $\Omega\subset \mathbb R^{n+1}_+$ be a bounded  open set whose boundary can be partitioned into $\partial^0 \Omega \subset \partial \mathbb R^{n+1}_+$ and $\partial^+\Omega \subset \partial \mathbb R^{n+1}_+$ both of which having positive measure. Let $\;u\in H^1(\Omega;\mathbb{R}^N)$ be such that $|u|=1$ $\mathscr{H}^n$-a.e. on $\partial^0 \Omega$. 
We say that $u$ is {\it weakly harmonic  in $\Omega$ with respect to the (partially) free boundary condition $u(\partial^0\Omega)\subset \mathbb{S}^{N-1}$} if  
\begin{equation}\label{weakformharmmap}
\int_\Omega \nabla u \cdot \nabla \Phi\,dx=0
\end{equation}
for all $\Phi \in H^1(\Omega;\mathbb{R}^N)\cap L^\infty(\Omega)$ compactly supported  $\Omega\cup\partial^0 \Omega$ and satisfying 
\begin{equation}\label{condtestharmon}
\Phi(x)\in T_{u(x)}\mathbb{S}^{N-1} \quad\text{$\mathscr{H}^n$-a.e. on $\partial^0 \Omega$}\,.
\end{equation}
In short, we shall say that $u$ is a {\it weak $(\mathbb{S}^{N-1}, \partial^0\Omega)$-boundary  harmonic map in $\Omega$}.  
\end{definition}

In view of the discussion above, let us mention that $(\mathbb{S}^{m-1}, \partial^0\Omega)$-boundary harmonic maps belong to a larger class of harmonic maps known in the literature as {\it harmonic maps with partially free boundary}, see \cite{BaldesMM1982,DuzaarSteffenAA1989,DuzaarSteffen1989JRAM,HardtLinCPAM1989,GulliverJostJRAM1987,DuzaarGrotowski94,sch} and references therein. In most studies, one considers a smooth compact Riemannian manifold $\mathcal{M}$ without boundary (that we can assume to be isometrically embedded in some Euclidean space by the Nash embedding Theorem), and $\mathcal{N}$ a smooth closed submanifold of $\mathcal{M}$. The boundary portion $\partial^0\Omega$ is called {\it the partially free boundary}, and $\mathcal{N}$ is {\it the supporting manifold}. Then $\mathcal{M}$-valued (weak) harmonic maps in $\Omega$ with the partially free boundary condition $u(\partial^0\Omega)\subset\mathcal{N}$ are defined as critical points of the Dirichlet energy  under the constraints  $u(x)\in\mathcal{M}$ for a.e. $x\in\Omega$ and $u(x)\in\mathcal{N}$ for $\mathscr{H}^n$-a.e. $x\in\partial^0\Omega$. For $(\mathbb{S}^{N-1}, \partial^0\Omega)$-boundary harmonic maps, we may consider the submanifold 
$\mathcal{N}=\mathbb{S}^{N-1}$ of the target $\mathcal{M}=\mathbb R^N$. However, to apply known results on harmonic maps, the compactness of $\mathcal{M}$ is usually required. A way to avoid this problem is to consider {\it bounded} $(\mathbb{S}^{N-1}, \partial^0\Omega)$-boundary harmonic maps, noticing that $\mathbb{S}^{N-1}$ can be viewed as a submanifold of a flat torus $\mathcal{M}=\mathbb{R}^N/r\mathbb{Z}^N$ with factor $r>2$.  

By a small adaptation of an argument introduced in \cite{millotSire}, we know that the weak $(\mathbb{S}^{N-1}, \partial^0\Omega)$-boundary harmonic map system can be reformulated as the following system considered in the distributional sense
\begin{equation}\label{eqweakFBCharm}
\begin{cases}
\Delta u = 0 & \text{in $\Omega$}\,,\\[8pt]
\displaystyle \frac{\partial u}{\partial\nu} + \mu \ \bot  \ T_u\,\mathbb{S}^{N-1} & \text{in $H^{-1/2}(\partial^0 \Omega)$}\,.
\end{cases}
\end{equation}
A classical result shows that the Dirichlet-to-Neumann map of an harmonic extension, considered as a map from $H^{1/2}$ into $H^{-1/2}$,  is precisely the root of the Laplacian. We can relate half-harmonic maps into $\mathbb{S}^{N-1}$  to $(\mathbb{S}^{N-1}, \partial^0\Omega)$-boundary harmonic maps and prove the following result (see \cite{millotSire}):
  
 \begin{proposition}[\bf Criticality transfer]\label{fracweakharm}
 Let $v\in \widehat H^{1/2}(\omega;\mathbb{R}^{N})$  be a weak 1/2-harmonic map into $\mathbb{S}^{N-1}$ in~$\omega$,  
 and let $v^{ext}$ be its harmonic extension to $\mathbb{R}^{n+1}_+$ given by the convolution with the Poisson kernel.
 Then $v^{ext}$ is  a weak $(\mathbb{S}^{N-1}, \partial^0\Omega)$-boundary  harmonic map in $\Omega$ for every admissible bounded  open set $\Omega\subset \mathbb R^{n+1}_+$  such that $\overline{\partial^0\Omega}\subset \omega$. 
  \end{proposition}
  
  This reformulation provides a useful dictionnary between half-harmonic maps and a distinguished class of harmonic maps with free boundary. Given the tremendous progress on the analysis of integral equations in the last decades, such reformulation has been paramount in proving fine results for free boundary harmonic mappings. See also Jost \cite{JostFB} for additional estimates.
  \vspace{0.2cm}
  
  The proof of Theorem \ref{thFBIntro} is then  a straightforward application of the previous Proposition \ref{fracweakharm}, the equation derived in \cite{millotSire} (see also \cite{MazSch18}) for instance and our Theorem \ref{cor2}.

We would like to emphasize that using Caffarelli-Silvestre extension argument, as in \cite{MilPegSch21}, one can derive as well some consequences of our estimates in the regime $s\in (0,1)$ in terms of {\sl weighted} harmonic maps with free boundary, which have been investigated also in \cite{Millot-Sire-Yu} (see also e.g. \cite{armin1}).

We now turn to time-dependent problems. We recall the distributional flow associated to half-harmonic maps: Look for $u:\mathbb R^n\times (0,\infty)\to \mathbb S^{N-1}$ that solves
\begin{equation*}
\begin{cases}
(\partial_t u+(-\Delta)^{\frac12} u)(x,t) \perp T_{u(x,t)} \mathbb S^{N-1} & \ x\in \mathbb R^n,\ t>0,\\
u(x,t)=u_0(x) & \ x\in \mathbb R^n, \ t=0.
\end{cases}
\end{equation*}

    Recently, Struwe in \cite{plateau} building on works by Wettstein \cite{Wet22} uncovered that the gradient flow of half-harmonic maps is actually a reformulation of a natural flow associated to the Plateau problem. More precisely, the flow considered in \cite{plateau}: consider an harmonic function $\tilde u(t)\in H^1(B, \mathbb R^K)$ where $B$ is the unit ball in $\mathbb R^2$ such that its traces $u(t)$ on $\partial B = \mathbb S^1$ belong to $H^{1/2}(\mathbb S^1, \mathcal N)$ for all $t>0$ and where $\mathcal N$ is a closed manifold sitting in $\mathbb R^K$ solving 
    \begin{equation}\label{eq:plateau}
    d\pi_{\mathcal N}(u)(u_t+\partial_r u)=u_t+d\pi_{\mathcal N}(u)\partial_r u=0,\,\,\,\text{on}\,\,\, \mathbb S^1 \times [0,\infty) 
    \end{equation} 
    where $\partial_r u$ is the Neumann derivative on the circle, and $\pi_{\mathcal N}:\mathcal N_\rho \to \mathcal N$ is the smooth nearest-neighbor projection from a $\rho-$neighbourhood of $\mathcal N$ into $\mathcal N$. In the special case when $\mathcal N=\mathbb S^{N -1}$ is an Euclidean sphere, as noted in \cite{plateau}, the system \eqref{eq:plateau} is the harmonic extension of the flow 
    $$
    u_t + (-\Delta)^{1/2} u \perp T_{u(x)} \mathbb S^{N -1},
    $$ 
 
    which is as already mentioned our original system. As for the previous stationary case, our results give estimates for the Plateau flow of Struwe, whenever the target is the round sphere.

\section*{Acknowledgments}
K.K. is supported by the Deutsche Forschungsgemeinschaft (DFG, German Research Foundation) under Germany's Excellence Strategy - EXC-2047/1 - 390685813. S.N. is supported by the Deutsche Forschungsgemeinschaft (DFG, German Research Foundation) - SFB 1283/2 2021 - 317210226. Y.S. is partially supported by NSF DMS grant $2154219$, " Regularity {\sl vs} singularity formation in elliptic and parabolic equations".

\bibliographystyle{alpha}
\bibliography{gp-nonlocal} 

\end{document}